\documentclass[10pt]{article}

\usepackage{latexsym,amsfonts,amsthm,amsmath,amssymb,cases}
\usepackage{cite}
\usepackage{enumerate}
\usepackage{graphicx}
\usepackage{pifont}
\usepackage{float}
\usepackage{caption}
\usepackage{subfigure}
\usepackage{epstopdf}
\usepackage{amsmath}
\usepackage{color}
\ifpdf
  \DeclareGraphicsExtensions{.eps,.pdf,.png,.jpg}
\else
  \DeclareGraphicsExtensions{.eps}
\fi

\usepackage{amsopn}

\newtheorem{remark}{Remark}

\numberwithin{figure}{section}
\newtheorem{theorem}{Theorem}[section]

\newtheorem{lm}{Lemma}[section]

\newtheorem{assumption}{Assumption}[section]

\newtheorem{defn}{Definition}[section]

\newtheorem{Al}{Algorithm}[section]
\newcounter{nextauthor}
\def\mathrm{\mbox}

\numberwithin{equation}{section}

\begin{document}
\title{{\Large\bf{Stochastic differential variational inequalities with applications}}\thanks{This work was supported by the National Natural Science Foundation of China (11671282, 12171339).}}
\author{Yao-Jia Zhang, Tao Chen, Nan-jing Huang\thanks{Corresponding author. E-mail address: nanjinghuang@hotmail.com; njhuang@scu.edu.cn}
and Xue-song Li
 \\{\small Department of Mathematics, Sichuan University, Chengdu, Sichuan, P.R. China}}
\date{ }
\maketitle

\noindent{\bf Abstract}: In this paper, we introduce and study a stochastic differential variational inequality (SDVI) which consists of a stochastic differential equation and a stochastic variational inequality. We obtain the existence and uniqueness of the solutions for SDVI by using the iteration method and Gronwall's inequality. Moreover, we show the convergence of Euler scheme for solving SDVI under some mild conditions. Finally, we apply the obtained results to solve the electrical circuits with diodes and the collapse of the bridge problems in stochastic environment.

\maketitle

\section{Introduction}
Pang and Stewart first systematically investigated the differential variational inequality (DVI) in \cite{Pang2008Differential}. They obtained solution
existence results and studied the numerical analysis of DVI by using a Euler time-stepping method. They further showed the sensitivity theorems and Newton approximation methods for DVI in \cite{Pang2009Solution} under some mild conditions.
Since then, various theoretical results, algorithms and applications concerned with DVIs have been studied by many authors under different conditions (see, for example, \cite{Han2009Convergence,Chen2012Computational,Wu2020A,Chen2013Convergence,Wang2018Existence,Wang2016A,Li2010Differential,Weng2021Rothe,Li2010Differential, Pang2008Differential,Zeng2018A,Weng2021Rothe,Brogliato2020Dynamical} and the references therein).

 As pointed out by Pang and Stewart \cite{Pang2008Differential}, DVI provides a powerful mathematical paradigm for modeling many real problems. By taking into account some stochastic environmental effects in DVIs, this paper investigates the following SDVI:
\begin{equation}\label{S1}
\left\{
\begin{aligned}
&dx(t)=f(t,x(t),u(t))dt+g(t,x(t),u(t))dB_t, \, t\in [0,T], \, \Gamma(x(0),x(T))=0,\\
&\left\langle F(t,\omega,x(t,\omega),u(t,\omega)),v-u(t,\omega) \right\rangle \geq 0, \, \forall v\in K, \, a.e. \, t\in [0,T], \, a.s. \, \omega\in \Omega,
\end{aligned}
\right.
\end{equation}
where $B_t$ is an $l$-dimensional standard Brownian motion, $K$ is a closed convex subset in $\mathbb{R}^m$, $f:[0,T]\times \mathbb{R}^n \times \mathbb{R}^m \rightarrow \mathbb{R}^n$, $g:[0,T]\times \mathbb{R}^n \times \mathbb{R}^m \rightarrow \mathbb{R}^{n\times l}$, $F: [0,T]\times \Omega \times \mathbb{R}^n \times K \rightarrow \mathbb{R}^m$, and $\Gamma: \mathbb{R}^n \times \mathbb{R}^n \rightarrow \mathbb{R}^n$ are measurable functions. As an extension of DVI under random environmental effects, SDVI \eqref{S1} provides a powerful mathematical tool for describing many real problems arising in engineering sciences, economics, biology such as stochastic electrical circuits, stochastic earthquake engineering, stochastic spatial price equilibrium problems, stochastic differential Nash games and stochastic batch fermentation problems (see, for example, \cite{Bensoussan2007Applied,Bensoussan2006Stochastic,Chen2012Computational,Chen2013Convergence,Pang2008Differential,Weng2021Rothe,
Zeng2018A,Yi2019A,Chen2021A,Raghunathan2006Parameter,Li2015Differential}).

 It is worth noting that there have been numerous studies concerning the special cases of SDVI \eqref{S1} under different conditions; for example, stochastic differential algebraic equations \cite{Nguyen2012Lyapunov,Milano2013A,Avaji2019Stability}, stochastic differential complementarity problems \cite{Random2018Tran}, stochastic differential inclusions  \cite{Yang2013Dynkin,Gassous2012Stochastic,Buckdahn2015Stochastic} and stochastic variational inequalities \cite{Bensoussan2007Applied,Bensoussan2006Stochastic}. Nevertheless, to the best of authors' knowledge, there is few work addressing the existence and uniqueness of the solutions for SDVIs. Thus, it would be significant and interesting to obtain some conditions for ensuring the existence and uniqueness of the solutions for SDVI \eqref{S1}. The first purpose of this paper is to show the existence and uniqueness of the solutions for SDVI \eqref{S1} under some mild conditions.

 On the other hand, fruitful results on numerical method of stochastic differential equations (SDEs) have been obtained by many authors including \cite{Tocino2002Runge,Platen1999An,Khodabin2011Numerical,Burrage2007Numerical,Peter1995Numberical} in the earlier years and \cite{Liu2022A,Cao2015Numerical,Li2021Numerical,Chirima2020Numerical} more recently. However, there is few work concerned with the numerical approximation method for solving SDVI \eqref{S1}. Therefore, it would be significant and attractive to investigate numerical approximation method for solving SDVI \eqref{S1}. The second goal of this paper is to propose and investigate a kind of Euler scheme for solving SDVI \eqref{S1} with $\Gamma(x(0),x(T))=x(0)-x_0$. Motivated and spirited by Euler method for solving DVIs \cite{Pang2008Differential,Chen2012Computational,Chen2013Convergence,Han2009Convergence,Li2010Differential,Wang2016A,Wang2018Existence}  and SDEs \cite{Platen1999An,Yang2020Strong,Peter1995Numberical}, we divide time interval $[0,T]$ into $N$ subintervals:
\begin{align*}
	0=t_0<t_1<t_2<\ldots <t_N=T,
\end{align*}
where $t_{i+1}-t_{i}=h=\frac{T}{N}$ $(i=0,1,\cdots,N-1)$, and  construct Euler scheme for solving SDVI \eqref{S1} as follows:
\begin{equation}\label{S1-e}
\left\{
\begin{aligned}
&x_h(t_{i+1})=x_h(t_{i})+f(t_{i},x_h(t_{i}),u_h(t_{i}))(t_{i+1}-t_{i})\\
&\qquad \quad  \mbox{} +g(t_{i},x_h(t_{i}),u_h(t_{i}))(B_{t_{i+1}}-B_{t_{i}}),\\
&x_h(t)=x_h(t_{i})+f(t_{i},x_h(t_{i}),u_h(t_{i}))(t-t_{i-1})\\
&\qquad \quad  \mbox{} +g(t_{i},x_h(t_{i}),u_h(t_{i}))(B_{t}-B_{t_{i}}),\, \forall t\in [t_{i},t_{i+1}),\, x_h(0)=x_0,\\
&\left\langle F(t_{i},\omega,x_h(t_{i}),u_h(t_{i})),v-u_h(t_{i}) \right\rangle \geq 0, \; \forall v\in K, \; a.s. \; \omega\in \Omega,\\
&u_h(t)=u_h(t_{i}), \quad \forall t\in [t_{i},t_{i+1}).
\end{aligned}
\right.
\end{equation}
In fact, we can compute $u_h(0)$ from the stochastic variational inequality (SVI) in \eqref{S1-e} for a given $x_0$, and then obtain $x_h(t_1)$ from SDE in \eqref{S1-e} by using $x_0$ and $u_h(0)$. In the same way, we can compute $u_h(t_1)$ and $x_h(t_2)$ by using $x_h(t_1)$. Again, repeat the above process.

The investigation of SDVIs is still rarely explored and much is desired to be done. The current work is an attempt in this new direction.  The contributions of this paper are twofold.
One is to obtain the existence and uniqueness of the solutions for SDVI \eqref{S1} by employing the iteration method and Gronwall's inequality. The other is to prove the convergence of Euler scheme  \eqref{S1-e} for solving SDVI \eqref{S1} under some mild conditions.

The outline of this paper is structured as follows. The next section recalls some necessary notions, symbols and lemmas. After that in Section 3, we prove the unique existence of the solutions for SDVI \eqref{S1}. In Section 4, we show the convergence of Euler scheme constructed by \eqref{S1-e}. Finally, we give two applications to the electrical circuits with diodes and the collapse of the bridge problems in stochastic environment in Section 5.

\section{Preliminaries}
\noindent \setcounter{equation}{0}
This section recalls some necessary notions, symbols and lemmas.
\begin{defn}${}$ \cite{Kuo1990Introduction,Bauschke2011Convex,Facchinei2003Finite}
\begin{itemize}
\item Denote the norm and inner product of $\mathbb{R}^n$ (or $\mathbb{R}^m$) by $\|\cdot\|$ and $\left\langle \cdot ,\cdot \right\rangle$.
\item $B_t$ is an $l$-dimensional standard Brownian motion with probability distribution $N(0,t)$.
\item $\mathcal{F}_t=\sigma\left\{B_t: 0\le s \le t\right\}$ is a filtration generated by $B_t$.

\item $L^2(\Omega ,\mathbb{R}^n)=L^2(\Omega,\mathcal{F},\mathbb{P},\mathbb{R}^n)$, the set which contains all the $\mathbb{R}^n$-valued
 measurable square-integrable random variables, is a Hilbert space fitted out the norm $\|\cdot\|_{L^2}=(E|\cdot|^2)^{\frac{1}{2}}$.

\item $H_{[a,b]}=L^2_{ad}([a,b]\times \Omega,\mathbb{R}^n)$  is the space of all stochastic processes $f(t,\omega)$, $0 \leq a\leq b\leq T$ valued in $\mathbb{R}^n$ such that $f(t,\omega)$ is adapted to the filtration $\left\{\mathcal{F}_t\right\}$ and $ \int_{a}^{b} \mathbb{E}\|f(t,\omega)\|^2dt < \infty $. Moreover, $L^2_{ad}([a,b]\times \Omega,\mathbb{R}^n)$ is a Hilbert space with the inner product
$$
\left\langle u,v \right\rangle_{H_{[a,b]}} = \int_{a}^{b} \mathbb{E}[\langle u(t,\omega),v(t,\omega)\rangle]dt,  \quad \forall u , v \in  L^2_{ad}([a,b]\times \Omega,\mathbb{R}^n)
$$
and norm
$$
\|u\|_{H_{[a,b]}}=\left[\int_{a}^{b} \mathbb{E}\|u(t,\omega)\|^2dt\right]^{\frac{1}{2}},  \quad \forall u  \in  L^2_{ad}([a,b]\times \Omega,\mathbb{R}^n).
$$
\end{itemize}
\end{defn}

Let $K\subset \mathbb{R}^m$ be nonempty. For any given given interval $[a,b]\subset [0,T]$, set
$$U[a,b]=\left\{u\in L^2_{ad}([a,b]\times \Omega,\mathbb{R}^m): u(s,\omega)\in K,a.e. \; s\in [a,b], \; a.s. \; \omega\in \Omega\right\}$$
and
$$
U=\left\{u\in L^2(\Omega ,\mathbb{R}^n):u(\omega)\in K, a.s. \; \omega\in \Omega \right\}.
$$

\begin{lm}\label{ccs}
Let $K\subset \mathbb{R}^m$ be nonempty, closed and convex. Then, the following statements hold:
\begin{itemize}
\item[(i)]  for any given interval $[a,b]\subset [0,T]$, $U[a,b]$ is nonempty, closed and convex in $L^2_{ad}([a,b]\times \Omega,\mathbb{R}^m)$;
\item[(ii)] $U$ is nonempty, closed and convex in $L^2(\Omega ,\mathbb{R}^m)$.
\end{itemize}

\begin{proof}
	For any $u(s,\omega)\in U[a,b]$, we have $u(s,\omega)\in L^2_{ad}([a,b]\times \Omega,\mathbb{R}^m)$ and so it is $\mathcal{F}_s$-measurable for all $s\in [a,b]$. It follows from the fact $K\not=\emptyset$ that $U[a,b]\not=\emptyset$. Now we show that $U[a,b]$ is convex in $L^2_{ad}([a,b]\times \Omega,\mathbb{R}^m)$. Indeed, for any $u_1(s,\omega),u_2(s,\omega) \in U[a,b]$ and $0<\lambda<1$, one has $u_1(s,\omega), u_2(s,\omega)\in L^2_{ad}([a,b]\times \Omega,\mathbb{R}^m)$ and $u_1(s,\omega),u_2(s,\omega) \in K$ for a.e. $t\in [a,b]$, a.s. $\omega\in \Omega$. Since $K$ is nonempty closed convex, it is easy to check that
	$$\lambda u_1(s,\omega)+(1-\lambda)u_2(s,\omega) \in K, \; a.e. \; t\in [a,b], \; a.s. \; \omega\in \Omega$$
	and so $U[0,t]$ is convex.
	
	Next we prove that $U[a,b]$ is closed in $L^2_{ad}([a,b]\times \Omega,\mathbb{R}^m)$. In fact, let $\left\{u_n(s,\omega)\right\} \subset U[a,b]$ be a sequence such that $\|u_n(s,\omega)-u^*(s,\omega)\|_{H[a,b]} \to 0$. Then it is easy to see that $u^*(t,\omega)\in L^2_{ad}([a,b]\times \Omega,\mathbb{R}^m)$. Moreover, for any $s \in [a,b]$,
	$$\|u_n(s,\omega)-u^*(s,\omega)\|^2_{H[a,b]}= \int_a^b\mathbb{E}\|u_n(s,\omega)-u^*(s,\omega)\|^2ds \to 0,$$
	which implies that
	$$\|u_n(s,\omega)- u^*(s,\omega)\|\to 0, \; a.e. \; s\in [a,b], \; a.s. \; \omega\in \Omega.$$
	Now, the closedness of $K$ shows that $u^*(s,\omega) \in U[a,b]$ and so $U[a,b]$ is closed. Thus the statement $(i)$ holds. Similarily, we can prove $(ii)$.
\end{proof}

\end{lm}

\begin{lm}\cite{Brezis2011Functional}
Let  $H$ be a  Hilbert space and $K\subset H$ be nonempty, closed and convex. Then, for each $m\in H$, there exists a unique $u\in K$, named as the projection of $m$ onto $K$ and denoted by $u=P_K(m)$, such that
\begin{align*}
\|m-u\|_H=\min_{v\in K}\|m-v\|_H=dist(m,K).
\end{align*}
Moreover, $u=P_K(m)$ if and only if
$$ u\in K, \quad \left\langle m-u,v-u \right\rangle_H \leq 0,\quad \forall v\in K.$$
\end{lm}

\begin{lm}\label{lm-p}\cite{Migorski2013Nolinear}
Let $H$ be a Hilbert space, $K\subset H$ be nonempty, closed  and convex, and $T: H \rightarrow H $ be a mapping. Then $x \in K$ is a solution of the following variational inequality (VI):
$$\left\langle T(x),y-x \right\rangle_H \ge 0, \quad \forall y\in K$$
if and only if $x=P_K(I-\alpha T)(x)$, where $\alpha$ is a positive constant and $I$ is an identity operator.
\end{lm}

\begin{lm}\label{gGi}\cite{Brunner2004Collocation,Yang2020Strong}
If there exist a non-negative sequence $\left\{z_n\right\}_{n=1}^{N}$ and a non-negative and non-decreasing sequence $\left\{y_n\right\}_{n=1}^{N}$ such that
\begin{align*}
z_n\leq \frac{Q_A}{N^{\beta}}\sum_{i=1}^{n-1}\frac{z_i}{(n-i)^{1-\beta}}+y_n,\quad 1\leq n\leq N, 0<\beta<1, Q_A>0,
\end{align*}
then the sequence $\left\{z_n\right\}_{n=1}^{N}$ can be bounded from above by
\begin{align*}
z_n \leq y_n(1+\mathbb{G}_{\beta}[Q_A\Gamma(\beta)], 1\leq n\leq N,
\end{align*}
where $\Gamma(x)=2\int_0^{+\infty}t^{2x-1}e^{-t^2}dt$ is the $\Gamma$-function and $\mathbb{G}_{\beta}(x)=\sum\limits_{k=0}^{\infty}\frac{x^k}{\Gamma(\beta k+1)}$ with $x\in \mathbb{R}$.
\end{lm}

\begin{lm}\label{lm-sg}\cite{Kuo1990Introduction}
	Let $\left\{ y_n  \right\}_{n=1}^{\infty}$ be a sequence of functions in $L^1[a,b]$ satisfying equation
	\begin{align*}
		y_{n+1}(t)\leq \alpha + \beta \int_a^t y_n(s)ds, \quad \forall t\in[a,b],
	\end{align*}
	where $\alpha$ and $\beta$ are constants. If $y_1(t)\equiv c$ is a constant, then the following inequality holds for any $n \geq 1$:
	\begin{align*}
		y_{n+1}(t)\leq \alpha e^{\beta (t-a)} + c \frac{\beta^n(t-a)^n}{n!}.
	\end{align*}
\end{lm}

\begin{lm}\label{lm-conversion}
	For any fixed $x\in L^2_{ad}([a,b]\times \Omega,\mathbb{R}^n)$, if $u\in U[a,b]$, then the following SVI
	\begin{align}\label{svi-r}
		\langle F(t,\omega,x(t,\omega),u(t,\omega)),v-u(t,\omega) \rangle \geq 0,, \quad \forall v\in K, a.e. \; t\in [a,b], \; a.s. \; \omega\in \Omega
	\end{align}
	is equivalent to that, for all $v'\in U[a,b]$ with $t\in (0,T]$,
	\begin{align}\label{svi-h}
		\left\langle \tilde{F}(x,u),v'-u \right\rangle_{H_{[a,b]}} \geq 0,
	\end{align}
	where
$
\tilde{F}: L^2_{ad}([a,b]\times \Omega,\mathbb{R}^n) \times U[a,b] \to L^2_{ad}([a,b]\times \Omega,\mathbb{R}^m)
$
is defined by
$$
\tilde{F}(x,u)(s,\omega):=F(s,\omega,x(s,\omega),u(s,\omega))
$$
for all $(x,u)\in L^2_{ad}([a,b]\times \Omega,\mathbb{R}^n) \times U[a,b]$, for all $s\in [a,b]$, and all $\omega \in \Omega$.
\end{lm}

\begin{proof}
	We first show that \eqref{svi-r} leads to \eqref{svi-h}. If \eqref{svi-h} false, then there exists a $v_0\in U[a,b]$ such that
	\begin{align*}
		\mathbb{E}\int_a^b \langle F(s,\omega,x(s,\omega),u(s,\omega)), v_0(s,\omega)-u(s,\omega)  \rangle ds <0, \quad \forall (s,\omega)\in [a,b]\times \Omega,
	\end{align*}
	which implies that
	\begin{align*}
		\Lambda^0:=\left\{ (s,\omega)\in [a,b]\times \Omega: \langle F(s,\omega,x(s,\omega),u(s,\omega)), v_0(s,\omega)-u(s,\omega)  \rangle <0 \right\} \neq \emptyset
	\end{align*}
	with $\int_{\Omega}\int_a^b I_{\Lambda^0}dsdP(\omega)>0$. From the  definition of $U[a,b]$, we know $v_0(s,\omega) \in K$ for a.e. $t\in [a,b]$ and a.s. $\omega \in \Omega$.
	Thus, there exists $v'\in K$ such that
	\begin{align*}
		\langle F(s,\omega,x(s,\omega),u(s,\omega)), v'-u(s,\omega) \rangle <0, \forall (s,\omega) \in \Lambda^0,
	\end{align*}
	which contradicts with \eqref{svi-r}.
	
	Next we prove that \eqref{svi-h} leads to \eqref{svi-r}. On the contrary, suppose that \eqref{svi-r} fails. Then there exists a $u_0 \in K$ and $\epsilon>0$ such that
	\begin{align*}
		\Lambda_{\epsilon}:=\left\{ (s,\omega)\in [a,b]\times \Omega: \langle F(s,\omega,x(s,\omega),u(s,\omega)), v_0(s,\omega)-u(s,\omega)  \rangle <\epsilon \right\}
	\end{align*}
	with $\int_{\Omega}\int_a^b I_{\Lambda_{\epsilon}}dsdP(\omega)>0$. Since $\langle F(s,\omega,x(s,\omega),u(s,\omega)), v_0(s,\omega)-u(s,\omega)\rangle$ is $\mathcal{F}_s$-adapted, so is  $I_{\Lambda_{\epsilon}}(s,\omega)$. Let
	$$
	\tilde{u}(s,\omega)=u_0I_{\Lambda_{\epsilon}}(s,\omega)+u(s,\omega)I_{\Lambda_{\epsilon}^c}(s,\omega),\quad \forall (t,\omega)\in [a,b]\times \Omega.
	$$
	Clearly, $\tilde{u} \in U[a,b]$. By taking $v=\tilde{u}$ in \eqref{svi-h} , one has
	\begin{align*}
		&\quad \; \mathbb{E}\int_a^b \langle  F(s,\omega,x(s,\omega),u(s,\omega)),\tilde{u}(s,\omega)-u(s,\omega) \rangle ds\\
		&=\int_{\Omega}\int_a^b I_{\Lambda_{\epsilon}}(s,\omega)\langle   F(s,\omega,x(s,\omega),u(s,\omega)),u_0-u(s,\omega)\rangle dsdP(\omega)\\
		&\leq - \epsilon\int_{\Omega}\int_a^b I_{\Lambda_{\epsilon}}dsdP(\omega)<0,
	\end{align*}
	which is a contradiction.
\end{proof}

\section{Existence and uniqueness of the solutions for SDVI}

In this section, we prove the existence and uniqueness of the solutions for SDVI. To this end, we need the following assumption and lemmas.

\begin{assumption} \label{a-e-1} Suppose that there exist a few positive constants $C$, $L_f$, $K_1$, $K_2$, $L_g$, $L_F$ with $L_F >C$ such that
	\begin{enumerate}[($\romannumeral1$)]
		\item $\|f(t,x,u)\|\leq K_1(1+\|x\|+\|u\|)$;
		\item $\|g(t,x,u)\|_{\mathbb{R}^{n\times m}} \leq K_2(1+\|x\|+\|u\|)$;
		\item $\|f(t_1,x_1,u_1)-f(t_2,x_2,u_2)\|\le L_f(|t_1-t_2|+\|x_1-x_2\|+\|u_1-u_2\|)$;
		\item $\|g(t_1,x_1,u_1)-g(t_2,x_2,u_2)\|_{\mathbb{R}^{n\times m}} \le L_g(|t_1-t_2|+\|x_1-x_2\|+\|u_1-u_2\|)$;
		\item $\|\tilde{F}(x'_1,u'_1)-\tilde{F}(x'_2,u'_2)\|_{H[0,t]} \le L_F(\|x'_1-x'_2\|_{H[0,t]}+\|u'_1-u'_2\|_{H[0,t]})$;
		\item $\left\langle \tilde{F}(x',u'_1)-\tilde{F}(x',u'_2),u'_1-u'_2 \right\rangle_{H[0,t]} \ge C\|u'_1-u'_2\|_{H[0,t]} ^2$,
	\end{enumerate}
	for all $t,t_1,t_2 \in [0,T]$, all $x,x_1,x_2 \in \mathbb{R}^n$, all $u,u_1,u_2 \in \mathbb{R}^m$, all $x',x'_1,x'_2 \in  L^2_{ad}([0,t]\times \Omega,\mathbb{R}^n)$, all $u'_1,u'_2\in U[0,t]$.
\end{assumption}

\begin{lm}\label{lm-ex-u}
For any fixed $x\in L^2_{ad}([0,b]\times \Omega,\mathbb{R}^n)$ $(0<b\leq T)$, if $\tilde{F}$ satisfies the conditions $(v)$ and $(vi)$ in Assumption \ref{a-e-1}, then there exists a unique $u\in U[0,b]$ such that
$$
\langle F(t,\omega,x(t,\omega),u(t,\omega)),v-u(t,\omega) \rangle \geq 0, \quad \forall v\in K, a.e. \; t\in [0,b], \; a.s. \; \omega\in \Omega.
$$
\end{lm}

\begin{proof}
	From Lemma \ref{lm-conversion}, we know that the problem is equivalent to finding a unique $u\in U[0,b]$ such that
	$$
	\left\langle \tilde{F}(x,u),v'-u \right\rangle_{H_{[0,b]}} \geq 0, \quad \forall v'\in U[0,b].
	$$
	Consider the following iterative sequence:
	$$
	u_{n+1}=P_{U[0,b]}[u_n-\rho\tilde{F}(x,u_n)], \quad \rho >0.
	$$
	By the conditions $(v)$ and $(vi)$ in Assumption 3.1, one has
	\begin{align*}
		\|u_{n+1}-u_{n}\|^2_{H_{[0,b]}}&=\|P_{U[0,b]}[u_n-\rho\tilde{F}(x,u_n)]-P_{U[0,b]}[u_{n-1}-\rho\tilde{F}(x,u_{n-1})]\|^2_{H_{[0,b]}}\\
		&\leq \|u_n-u_{n-1}-\rho \tilde{F}(x,u_n)+\rho \tilde{F}(x,u_{n-1})\|^2_{H_{[0,b]}}\\
		&=\|u_n-u_{n-1}\|^2_{H_{[0,b]}}+\rho^2\|\tilde{F}(x,u_{n-1})-\tilde{F}(x,u_n)\|^2_{H_{[0,b]}}\\
		&\quad \: -2\rho \langle \tilde{F}(x,u_{n})-\tilde{F}(x,u_{n-1}),u_n-u_{n-1} \rangle_{H_{[0,b]}}\\
		&\leq (1-2\rho C +\rho^2 L^2_F)	\|u_{n+1}-u_{n}\|^2_{H_{[0,b]}},
	\end{align*}
	where $ 0<\rho <\frac{2C}{L^2_F}$, which implies that there exists a unique $u^*\in L^2_{ad}([0,b]\times \Omega,\mathbb{R}^m)$ such that $u^*=P_{U[0,b]}[u^*-\rho\tilde{F}(x,u^*)]$. It follows from Lemma \ref{lm-p} that
	$$
	\left\langle \tilde{F}(x,u^*),v'-u^* \right\rangle_{H_{[0,b]}} \geq 0, \quad \forall v'\in U[0,b].
	$$
	This ends the proof.
\end{proof}

\begin{lm}\label{lm-u-bd}
	For each $j=1,2$ and fixed $x_j \in L^2_{ad}([0,T]\times \Omega,\mathbb{R}^n)$, under the conditions $(v)$ and $(vi)$ in Assumption \ref{a-e-1}, there exists a unique $u_j\in [0,T]$ such that
	$$
	\langle F(t,\omega,x_j(t,\omega),u_j(t,\omega)),v-u_j(t,\omega) \rangle \geq 0, \quad \forall v\in K, a.e. \; t\in [0,T], \; a.s. \; \omega\in \Omega.
	$$
	Moreover, there exists a constant $M'>0$ such that
	\begin{align*}
		\mathbb{E}\int_0^t\|u_1(s,\omega)-u_2(s,\omega)\|^2ds\leq M' \mathbb{E}\int_0^t\|x_1(s,\omega)-x_2(s,\omega)\|^2ds,\forall t\in[0,T].
	\end{align*}
\begin{proof}
	For each $j=1,2$, it follows from Lemma \ref{lm-ex-u} that there exists a unique $u_j \in U[0,T]$ such that
	$$
	\langle F(s,\omega,x_j(s,\omega),u_j(s,\omega)),v-u_j(s,\omega) \rangle \geq 0, \quad \forall v\in K, a.e. \; s\in [0,T], \; a.s. \; \omega\in \Omega.
	$$
	By Lemmas \ref{lm-p} and \ref{lm-conversion},  we obtain $u_j =P_{U[0,t]}[u_j-\rho \tilde{F}(x_j,u_j)]$ for all $t\in [0,T]$, $j=1,2$.
	Now the conditions $(v)$ and $(vi)$ in Assumption \ref{a-e-1} implies that
	\begin{align}\label{lm-u-bd-1}
		&\|u_1-u_2+\rho \tilde{F}(x_1,u_2)-\rho \tilde{F}(x_1,u_1)\|^2_{H[0,t]}  \\
		=&\|u_1-u_2\|^2_{H[0,t]}+\rho^2\|\tilde{F}(x_1,u_2)-\rho \tilde{F}(x_1,u_1)\|^2_{H[0,t]}\nonumber \\
		&-2\rho \langle \tilde{F}(x_1,u_1)- \tilde{F}(x_1,u_2),u_1-u_2  \rangle_{H[0,t]}\nonumber \\
		\leq& (1-2\rho C+\rho^2L^2_F)\mathbb{E}\|u_1-u_2\|^2_{H[0,t]},\nonumber
	\end{align}
where $0 < \rho < \frac{2C}{L^2_F}$. Applying \eqref{lm-u-bd-1} and the nonexpansiveness of $P_{U[0,t]}$, one has
\begin{align*}
	&\|u_1-u_2\|_{H[0,t]}\\
	=&\|P_{U[0,t]}[u_1-\rho\tilde{F}(x_1,u_1)]-P_{U[0,t]}[u_1-\rho\tilde{F}(x_1,u_1)]\|_{H[0,t]}\nonumber\\
	\leq& \|u_1-u_2-\rho\tilde{F}(x_1,u_1)+\rho\tilde{F}(x_2,u_2)\|_{H[0,t]}\nonumber \\
	=&\|u_1-u_2+\rho\tilde{F}(x_2,u_2)-\rho\tilde{F}(x_1,u_2)+rho\tilde{F}(x_1,u_2)-\rho\tilde{F}(x_2,u_2)\|_{H[0,t]}\nonumber \\
	\leq& \|u_1-u_2+\rho\tilde{F}(x_2,u_2)-\rho\tilde{F}(x_1,u_2)\|_{H[0,t]}+\rho \| \tilde{F}(x_1,u_2)-\tilde{F}(x_2,u_2)\|_{H[0,t]}\nonumber \\
	\leq& \sqrt{1-2\rho C+\rho^2L^2_F}\|u_1-u_2  \|_{H[0,t]} +\rho L_F \| x_1-x_2 \|_{H[0,t]},
\end{align*}
	which implies that
\begin{align*}
	\mathbb{E}\int_0^t\|u_1(s,\omega)-u_2(s,\omega)\|^2ds\leq M' \mathbb{E}\int_0^t\|x_1(s,\omega)-x_2(s,\omega)\|^2ds,\forall t\in[0,T],
\end{align*}
where $M'=\frac{\rho^2L^2_F}{(1-\sqrt{1-2\rho C+\rho^2L^2_F})^2}$.
\end{proof}

\end{lm}

In the sequel, we omit $\omega$ in $x(t,\omega)$ and $
u(t,\omega)$ sometimes without ambiguity.

Let $SOL(U[0,T],F(t,\omega,x(t),u(t)))$ be the set of solutions for following SVI: find $u(t) \in U[0,T]$ such that
\begin{align}\label{S2}
	\left\langle F(t,\omega,x(t),u(t)),v-u(t) \right\rangle \geq 0, \quad \forall v\in K, a.e. \; t\in [0,T], \; a.s. \; \omega\in \Omega.
\end{align}
A pair $(x(t),u(t))$ is called a Carath\'{e}odory solution of problem SDVI \eqref{S1} with $\Gamma(x(0),x(T))=x(0)-x_0$ if and only if $x(t)$ belongs to $L^2_{ad}([0,T]\times \Omega,\mathbb{R}^n)$ such that, for any $t\in [0,T]$,
$$
dx(t)=f(t,x(t),u(t))dt+g(t,x(t),u(t))dB_t,\quad x(0)=x_0
$$
and $u(t)\in SOL(U[0,T],F(t,\omega,x(t),u(t)))$.

Now we give the main result of this section as follows.

\begin{theorem}\label{thm-e}
	Under Assumption \ref{a-e-1}, SDVI \eqref{S1} admits a unique Carath\'{e}odory solution $(x(t),u(t))$.
\end{theorem}

\begin{proof}
 Taking $x^{(1)}_t=x_0$ for all $t\in[0,T]$, it follows from Lemmas \ref{lm-p}, \ref{lm-conversion} and \ref{lm-ex-u} that there exists a unique $u^{(1)}_t\in U[0,T]$ such that $u^{(1)}_t=P_{U[0,T]}[u^{(1)}_t-\rho \tilde{F}(x^{(1)}_t,u^{(1)}_t)]$. Let $x^{(2)}_t=x_0+\int_0^t f(s,x^{(1)}_s,u^{(1)}_s)ds + \int_0^t g(s,x^{(1)}_s,u^{(1)}_s) dB_s$.
 Again, by Lemmas \ref{lm-p}, \ref{lm-conversion} and \ref{lm-ex-u}, we know there exists a unique $u^{(2)}_t\in U[0,T]$ such that $u^{(2)}_t=P_{U[0,T]}[u^{(2)}_t-\rho \tilde{F}(x^{(2)}_t,u^{(2)}_t)]$.
 Repeating this process, we can construct the sequence $\left\{ (x^{(n)}_t, u^{(n)}_t) \right\}$ satisfying
\begin{equation}\label{interation}
	\left\{
	\begin{aligned}
		&u^{(n)}_t=P_{U[0,T]}\left(u^{(n)}_t-\rho \tilde{F}(x^{(n)}_t,u^{(n)}_t)\right),\\
		&x^{(n+1)}_t=x_0+\int_0^t f(s,x^{(n)}_s,u^{(n)}_s)ds + \int_0^t g(s,x^{(n)}_s,u^{(n)}_s) dB_s,
	\end{aligned}
	\right.
\end{equation}
where $0 < \rho < \frac{2C}{L^2_F}$, $n=1,2,\cdots$.

First we show that $(x^{(n+1)}_t, u^{(n+1)}_t)\in L^2_{ad}([0,T]\times \Omega,\mathbb{R}^n)\times {U[0,T]}$ by the induction. Obviously, $x^{(1)}_t=x_0$  for all $t\in [0,T]$ and $u^{(1)}_t=P_{U[0,T]}[u^{(1)}_t-\rho \tilde{F}(x^{(1)}_t,u^{(1)}_t)]$ lead to $(x^{(1)}_t, u^{(1)}_t)\in L^2_{ad}([0,T]\times \Omega,\mathbb{R}^n)\times {U[0,T]}$. Assume by the induction that $(x^{(n)}_t, u^{(n)}_t)\in L^2_{ad}([0,T]\times \Omega,\mathbb{R}^n)\times {U[0,T]}$. Then it follows from conditions $(i)$ and $(ii)$ in Assumption \ref{a-e-1} that
\begin{equation}\label{xn-1}
	\aligned
	&\quad \;\mathbb{E}\left[\int_0^{T}g(t,x^{(n)}_t,u^{(n)}_t)dB_t\right]^2=\mathbb{E}\left[\int_0^{T}\|g(t,x^{(n)}_t,u^{(n)}_t)\|_{\mathbb{R}^{n\times m}}^2dt\right]\\
	& \le 3K_2^2T+3K_2^2\mathbb{E}\left[\int_0^{T}\left(\|x^{(n)}_t\|^2+\|u^{(n)}_t\|^2\right)dt\right]<+\infty
	\endaligned
\end{equation}
and
\begin{equation}\label{xn-1*}
	\aligned
	&\quad \;\mathbb{E}\left[\int_0^{T}f(t,x^{(n)}_t,u^{(n)}_t)dt\right]^2\\
	&\le 3K_1^2T\mathbb{E}\left[\int_0^{T}\left(1+\|x^{(n)}_t\|^2+\|u^{(n)}_t\|^2\right)dt\right]<+\infty.
	\endaligned
\end{equation}
Moreover, we have
\begin{eqnarray}\label{xn-2}
	&&\quad \;\mathbb{E}\|x^{(n+1)}_t\|^2\\
	&&\le3\mathbb{E}\left[x_0^2+\left[\int_0^{T}f(t,x^{(n)}_t,u^{(n)}_t)dt\right]^2+\left[\int_0^{T}g(t,x^{(n)}_t,u^{(n)}_t)dB_t\right]^2\right].\nonumber
\end{eqnarray}
Combining \eqref{xn-1} to \eqref{xn-2}, we know that $\mathbb{E}\left[\int_0^{T}|x^{(n+1)}_t|^2dt\right]<\infty$ and so $x^{(n+1)}_t\in L^2_{ad}([0,T]\times \Omega,\mathbb{R}^n)$. By \eqref{interation} and the definition of $P_{U[0,T]}$, one has $u^{(n+1)}_t\in U[0,T]$.

Next we prove that $\left\{(x^{n}_t,u^{n}_t)\right\}$ is a Cauchy sequence in $L^2_{ad}([0,T]\times \Omega,\mathbb{R}^n) \times U[0,T]$. It follows from Cauchy-Schwartz's inequality and the conditions $(iii)$ and $(iv)$ in Assumption \ref{a-e-1} that
\begin{align}\label{thm-sdvi-x-1}
	&\quad \;\mathbb{E}\|x^{(n+1)}_t-x^{(n)}_t\|^2 \\
	&\le 2\mathbb{E}\Bigg[\left(\int_0^t\left(f(s,x^{(n)}_s,u^{(n)}_s)-f(s,x^{(n-1)}_s,u^{(n-1)}_s)\right) ds\right)^2\nonumber\\
	&\quad\;+ \left(\int_0^t\left(g(s,x^{(n)}_s,u^{(n)}_s)-g(s,x^{(n-1)}_s,u^{(n-1)}_s)\right) dB_s\right)^2\Bigg]\nonumber\\
	&\le4L_g^2
	\left(\int_0^t \mathbb{E}\|x^{(n)}_s-x^{(n-1)}_s\|^2 ds +\int_0^t \mathbb{E}\|u^{(n)}_s-u^{(n-1)}_s\|^2 ds\right)\nonumber\\
	&\quad\;+4t_0 L_f^2\left(\int_0^t \mathbb{E}\|x^{(n)}_s-x^{(n-1)}_s\|^2 ds +\int_0^t \mathbb{E}\|u^{(n)}_s-u^{(n-1)}_s\|^2 ds\right)\nonumber\\
	&=4(L_g^2+TL_f^2)\left(\int_0^t \left(\mathbb{E}\|x^{(n)}_s-x^{(n-1)}_s\|^2 + \mathbb{E}\|u^{(n)}_s-u^{(n-1)}_s\|^2\right) ds\right).\nonumber
\end{align}
On the other hand, from Lemma \ref{lm-u-bd}, we know there exists a constant $M'>0$, for all $t\in[0,T]$, such that
\begin{align}\label{thm-sdvi-u}
	\mathbb{E}\int_0^t\|u^{(n)}(s,\omega)-u^{(n-1)}(s,\omega)\|^2ds\leq M' \mathbb{E}\int_0^t\|x^{(n)}(s,\omega)-x^{(n-1)}(s,\omega)\|^2ds.
\end{align}
Combining \eqref{thm-sdvi-x-1} and \eqref{thm-sdvi-u}, for all $t\in[0,T]$, one has
\begin{align}\label{thm-sdvi-x-2}
	\mathbb{E}\|x^{(n+1)}_t-x^{(n)}_t\|^2 \leq \beta \mathbb{E}\int_0^t\|x^{(n)}(s,\omega)-x^{(n-1)}(s,\omega)\|^2ds,
\end{align}
where $\beta=4(L_g^2+TL_f^2)(1+M')$.
Moreover, the conditions $(i)$ and $(iii)$ in Assumption \ref{a-e-1} lead to
\begin{align}\label{thm-sdvi-u-2}
	\mathbb{E}\|x^{(2)}_t-x^{(2)}_t\|^2 \leq 6(K^2_1+K^2_2)\mathbb{E}\int_0^t(1+\|x_0\|^2+\|u^{(1)}_t\|^2)ds.
\end{align}
It follows from \eqref{thm-sdvi-x-2}, \eqref{thm-sdvi-u-2} and Lemma \ref{lm-sg} that
\begin{align*}
\mathbb{E}\|x^{(n+1)}_t-x^{(n)}_t\|^2 \leq \eta \frac{\beta^{n-1}t^n}{n!},
\end{align*}
which implies that
$$
\|x^{(n+1)}-x^{(n)}\|^2_{H[0,T]}\leq \eta \frac{\beta^{n-1}T^{n+1}}{n!} \rightarrow 0, \quad (n \rightarrow \infty),
$$
where $\eta=6(K^2_1+K^2_2)(1+\|x_0\|^2_{H[0,T]}+\| u^{(1)}_t\|^2_{H[0,T]})$.
Now Lemma \ref{lm-u-bd} implies that $\left\{(x^{n}_t,u^{n}_t)\right\}$ is a Cauchy sequence in $L^2_{ad}([0,T]\times \Omega,\mathbb{R}^n) \times U[0,T]$.

Therefore, there exists a couple $(x^*(t),u^*(t))\in L^2_{ad}([0,T]\times \Omega,\mathbb{R}^n) \times U[0,T]$ such that
$(x^{(n)}(t),u^{(n)}(t))\to (x^*(t),u^*(t))$. Furthermore, it follows from the continuity of $P_{U[0,T]}$ and dominated convergence theorem that
\begin{align*}
	\left\{
	\begin{aligned}
		u^*_t&=\lim \limits_{n\rightarrow\infty}P_{U[0,T]}\left(u^{(n)}_t-\rho \tilde{F}(x^{(n)}_t,u^{(n)}_t)\right)\\
		&=P_{U[0,T]}\bigg(u^*_t-\rho \tilde{F}(x^*_t,u^*_t)\bigg),\\
		x^*_t&=x_0+\lim \limits_{n\rightarrow\infty}\int_0^t f(s,x^{(n)}_s,u^{(n)}_s)ds + \lim \limits_{n\rightarrow\infty}\int_0^t g(s,x^{(n)}_s,u^{(n)}_s) dB_s\\
		&=x_0+\int_0^t f(s,x^*_s,u^*_s)ds+\int_0^t g(s,x^*_s,u^*_s) dB_s,
	\end{aligned}
	\right.
\end{align*}
which implies that the limit $(x^*_t,u^*_t)$ is a Carath\'{e}odory solution for SDVI \eqref{S1} with $\Gamma(x(0),x(T))=x(0)-x_0$ in $L^2_{ad}([0,T]\times \Omega,\mathbb{R}^n) \times U[0,T]$.

Finally, we show the uniqueness of the Carath\'{e}odory solutions for SDVI \eqref{S1} with $\Gamma(x(0),x(T))=x(0)-x_0$. Suppose there exist $(x_i,u_i)\in L^2_{ad}([0,T]\times \Omega,\mathbb{R}^n) \times U[0,T]$ $(i=1,2)$ satisfying
	\begin{align}\label{thm-sdvi-unique-1}
	\left\{
	\begin{aligned}
		u_i(t)&=P_{U[0,T]}\left(u_i(t)-\rho \tilde{F}(x_i(t),u_i(t))\right),\\
		x_i(t)&=x_0+\int_0^t f(s,x_i(s),u_i(s))ds + \int_0^t g(s,x_i(s),u_i(s))dB_s.
	\end{aligned}
	\right.
\end{align}
Then, similar to the proof of \eqref{thm-sdvi-x-2}, it follows from \eqref{thm-sdvi-unique-1} and the conditions $(i)$ to $(iv)$ in Assumption \ref{a-e-1} that
\begin{align*}
	\mathbb{E}\|x_1-x_2\|^2 \leq \beta \mathbb{E}\int_0^t\|x_1(s,\omega)-x_2(s,\omega)\|^2ds,\forall t\in[0,T],
\end{align*}
where $\beta=4(L_g^2+TL_f^2)(1+M')$. By Gronwall's inequality, we have $	\mathbb{E}\|x_1-x_2\|^2=0$. Now Lemma \ref{lm-u-bd} leads to
$$
\|x_1-x_2\|_{H[0,T]}=0, \quad \|u_1-u_2\|_{H[0,T]}=0.
$$
This show the uniqueness of the Carath\'{e}odory solutions.
\end{proof}

\section{Convergence of Euler scheme for solving SDVI}

In this section, we prove the convergence of Euler scheme  constructed by \eqref{S1-e}. To this end,  we need the following hypotheses and lemmas.

\begin{assumption} \label{a-e} Suppose that there exist positive constants $C$, $L_f$, $L_g$, and $L_F$ with $L_F >C$ such that
\begin{enumerate}[($\romannumeral1$)]
\item $\|f(t_1,x_1,u_1)-f(t_2,x_2,u_2)\|\le L_f(|t_1-t_2|+\|x_1-x_2\|+\|u_1-u_2\|)$;
\item $\|g(t_1,x_1,u_1)-g(t_2,x_2,u_2)\|_{\mathbb{R}^{n\times m}} \le L_g(|t_1-t_2|+\|x_1-x_2\|+\|u_1-u_2\|)$;
\item $\|F(t_1,\omega,x_1,u_1)-F(t_2,\omega,x_2,u_2)\| \le L_F(|t_1-t_2|+\|x_1-x_2\|+\|u_1-u_2\|)$, a.s. $\omega \in \Omega$;
\item $\left\langle F(t,\omega,x,u_1)-F(t,\omega,x,u_2),u_1-u_2 \right\rangle \ge C\|u_1-u_2\|^2$, a.s. $\omega \in \Omega$,
\end{enumerate}
for all $t, t_1,t_2 \in [0,T]$, all $x,x_1,x_2 \in \mathbb{R}^n$, all $u_1,u_2 \in \mathbb{R}^m$.
\end{assumption}

\begin{remark}\label{lgc}
Clearly, $(i)$ and $(ii)$ imply that $f(t,x,u)$ and $g(t,x,u)$ have linear growth in $x$ and $u$, i.e., there exist two constants $K_1>0$ and $K_2>0$ conforming to
$\|f(t,x,u)\|\le K_1(1+\|x\|+\|u\|)$ and $\|g(t,x,u)\|_{\mathbb{R}^{n\times m}} \le K_2(1+\|x\|+\|u\|)$ for all $(t,\omega) \in [0,T]\times \Omega$. It is worth noting that Assumption \ref{a-e-1} is a speacil case of Assumption \ref{a-e}.
\end{remark}

\begin{lm}\label{lm-h-e}
Let $F$ satisfy the conditions $(iii)$ and $(iv)$ in Assumption \ref{a-e}. Then $u_h(t)$ solves the following stochastic variational inequality
\begin{align}\label{sdviL^2-1}
&\mathbb{E}\int_{t_{i}}^t \left\langle F(t_i,\omega,x_h(t_i,\omega),u_h(s,\omega)),v(s,\omega)-u_h(s,\omega) \right\rangle ds \geq 0
\end{align}
for all $v \in U[t_i,t]$ and all $t\in[t_i,t_{i+1})$ with $i=0,1,\cdots,N$.

\end{lm}
\begin{proof}
First, for fixed $t\in[0,T]$, define a mapping
$
\bar{F}:L^2(\Omega,\mathbb{R}^n) \times U \to L^2(\Omega,\mathbb{R}^m)	
$
by
$$
\bar{F}(x,u)(\omega):=F(t,\omega,x(\omega),u(\omega)), \; \forall (x,u) \in L^2(\Omega,\mathbb{R}^n) \times U,\;\forall \omega \in \Omega.
$$
Then, for each $i=0,1,\cdots,N$, by the conditions $(iii)$ and $(iv)$ in Assumption \ref{a-e}, Lemma \ref{lm-ex-u} implies that there is a unique $u_h(t_i)$ satisfying
$$
\left\langle \bar{F}(x_h(t_i),u_h(t_i)),v-u_h(t_{i}) \right\rangle_{L^2(\Omega)} \geq 0, \; \forall v\in U.
$$
Let $u_h(t)=u_h(t_{i})$ for all $t\in [t_{i},t_{i+1})$ in \eqref{S1-e} with $i=0,1,\cdots,N$. Thus, by the definition of $U[a,b]$, we know that $u_h(s)\in U[t_i,t]$ for all $t\in [t_i,t_{i+1})$. Clearly, \eqref{S1-e} implies that
\begin{align*}
	\left\langle F(t_i,\omega,x_h(t_i,\omega),u_h(t,\omega)),v-u_h(t,\omega) \right\rangle \geq 0, \; \forall v\in K, \forall t\in [t_i,t_{i+1}), \; a.s. \; \omega\in \Omega.
\end{align*}
For any $t\in [t_i,t_{i+1})$, by taking $F(t,\omega,x(t,\omega),u(t,\omega))=F(t_i,\omega,x_h(t_i,\omega),u_h(t,\omega)))$ in Lemma \ref{lm-conversion}, we know that,
\begin{align*}
\mathbb{E}\int_{t_{i}}^t\left\langle F(t_i,\omega,x_h(t_i,\omega),u_h(s,\omega)),v(s,\omega)-u_h(s,\omega) \right\rangle ds \geq 0
\end{align*}
for all $v \in U[t_i,t]$ with $t \in[t_i,t_{i+1})$.
 This ends the proof.
\end{proof}

\begin{lm}\label{lm-ubd-1}
 Let $K\subset \mathbb{R}^n$ be closed convex. Under Assumption \ref{a-e}, there exists a constant $M \geq 0$ such that, for $i=1,\cdots,N$,
\begin{align*}
\|u_h(t_{i})\|_{L^2(\Omega)}\leq M (1+\|x_h(t_{i})  \|_{L^2(\Omega)}).
\end{align*}
\end{lm}
\begin{proof}
By Lemmas \ref{ccs}, \ref{lm-p} and \ref{lm-h-e}, there exists a unique $u_h(t_n) \in U$ satisfying
\begin{align}\label{ubd-u-1}
u_h(t_i)=P_{U}[u_h(t_i)-\rho \bar{F}(x_h(t_i),u_h(t_i))],
\end{align}
where $0 < \rho < \frac{2C}{L^2_F}$. It follows from \eqref{ubd-u-1} and the conditions (iii) and (iv) in Assumption \ref{a-e} that
\begin{align}\label{ubd-u-2}
&\quad \;\mathbb{E}\|u_h(t_{i})-u_h(t_0)+\rho F(t_0,\omega,x_h(t_{0}),u_h(t_{0}))-\rho F(t_0,\omega,x_h(t_{0}),u_h(t_{i}))\|^2\\
&\leq \mathbb{E}\|u_h(t_{i})-u_h(t_0)\|^2+\rho^2\mathbb{E}\|\bar{F}(x_h(t_{0}),u_h(t_{0}))-\rho \bar{F}(x_h(t_{0}),u_h(t_{i}))\|^2 \nonumber\\
&\mbox{} \quad -2\rho\mathbb{E}\Big[ \langle u_h(t_{i})-u_h(t_0), F(t_{0},\omega,x_h(t_{0}),u_h(t_{i}))-F(t_{0},\omega,x_h(t_{0}),u_h(t_{0}))\rangle\Big]\nonumber\\
&\leq (1-2\rho C+\rho^2L^2_F)\mathbb{E}\|u_h(t_{i})-u_h(t_0)\|^2.\nonumber
\end{align}
Applying \eqref{ubd-u-2} and the condition $(iii)$ in Assumption \ref{a-e}, we obtain
\begin{align*}
	&\mbox{}\quad \; \|u_h(t_{i})-u_h(t_0)\|_{L^2(\Omega)}\nonumber\\
	&=\left[\mathbb{E}\|P_U[u_h(t_{i})-\rho \bar{F}(x(t_{i}),u(t_{i}))]-P_U[u(t_0)-\rho \bar{F}(x_h(t_0),u_h(t_0))]\|^2\right]^{\frac{1}{2}}\nonumber\\
	&\leq \left[ \mathbb{E}\|u_h(t_{i})-u_h(t_0)+\rho F(t_0,\omega,x_h(t_0),u_h(t_0))-\rho F(t_i,\omega,x(t_{i}),u(t_{i}))\|^2\right]^{\frac{1}{2}}\nonumber\\
	&=\Big[\mathbb{E}\|u_h(t_{i})-u_h(t_0)+\rho F(t_0,\omega,x_h(t_0),u_h(t_0))-\rho F(t_0,\omega,x_h(t_{0}),u_h(t_{i}))\nonumber\\
	&\mbox{}\quad +\rho F(t_0,\omega,x_h(t_{0}),u_h(t_{i}))-\rho F(t_i,\omega,x_h(t_{n}),u_h(t_{i}))\|^2\Big]^{\frac{1}{2}}\nonumber\\
	&\leq \left[\mathbb{E}\|u_h(t_{i})-u_h(t_0)+\rho F(t_0,\omega,x_h(t_0),u_h(t_0))-\rho F(t_{0},\omega,x(t_{0}),u(t_{i}))\|^2\right]^{\frac{1}{2}}\nonumber\\
	&\mbox{}\quad+ \left[\mathbb{E}\|\rho F(t_{0},\omega,x(t_{0}),u(t_{i}))-\rho F(t_{i},\omega,x(t_{i}),u(t_{i}))\|^2\right]^{\frac{1}{2}}\nonumber\\
	&\leq \sqrt{(1-2\rho C+\rho^2L^2_F)}\|u_h(t_{i})-u_h(t_0)\|_{L^2(\Omega)}\\
	&\mbox{}\quad +\sqrt{2}\rho L_F ih+ \sqrt{2}\rho L_F \|x_h(t_{i})-x_h(t_0)\|_{L^2(\Omega)},
\end{align*}
which implies that
\begin{align}\label{ubd-u-3}
\|u_h(t_{i})-u_h(t_0)\|_{L^2(\Omega)}\leq M(ih+\|x_h(t_{i})-x_h(t_0)  \|_{L^2(\Omega)}),
\end{align}
where $M=\frac{\sqrt{2} \rho L_F}{1-\sqrt{(1-2\rho C+\rho^2L^2_F)}}$ and $u_h(t_0)=P_{U}[u_h(t_0)-\rho \bar{F}(x_h(t_0),u_h(t_0))]$.
The inequality \eqref{ubd-u-3} shows that there exists a constant $M>0$ such that
\begin{align}\label{ubd-u-4}
\|u_h(t_{i})\|_{L^2(\Omega)}\leq M (1+\|x_h(t_{i})  \|_{L^2(\Omega)}).
\end{align}
This ends the proof.
\end{proof}

At the end of this section, we show our main results as follows.

\begin{theorem}\label{th-1}
Let $K\subset \mathbb{R}^n$ be closed convex. Under Assumptions \ref{a-e}, one has
$$\|x_h-x\|_{H_{[0,T]}} \rightarrow 0, \quad \|u_h-u\|_{H_{[0,T]}} \rightarrow 0 \quad \mbox{as} \quad h\rightarrow 0,$$
where $(x_h(t),u_h(t))$ is given by \eqref{S1-e} and $(x(t),u(t))$ is a unique solution of SDVI \eqref{S1}. Moreover, there exists a constant $C>0$ such that
\begin{align*}
\|x_h-x\|^2_{H_{[0,T]}} \leq C[h^2+h], \quad
\|u-u_h\|_{H_{[0,T]}}^2\leq C[h^2+h].
\end{align*}
\end{theorem}
\begin{proof}
We first show the measurability of $x_h(t)$ and $u_h(t)$. In fact, define $\hat{s}=\hat{s}(s)$ by setting
\begin{equation}\label{s^}
\hat{s}(s)=\left\{
\begin{array}{ll}
t_i, & \quad s\in [t_i,t_{i+1});\\
0, &  \quad \mbox{else}.
\end{array}
\right.
\end{equation}
Then, it follows from \eqref{S1-e} and \eqref{s^} that
\begin{equation}\label{sdvi-h-s^}
\left\{
\begin{aligned}
&x_h(t)=x_0+\int_0^tf(\hat{s},x_h(\hat{s}),u_h(\hat{s}))ds+\int_0^tg(\hat{s},x_h(\hat{s}),u_h(\hat{s}))dB_s,\\
&\left\langle F(\hat{s},\omega,x_h(\hat{s}),u_h(\hat{s})),v-u_h(\hat{s}) \right\rangle \geq 0, \quad \forall v\in K, \; a.s. \; \omega\in \Omega.
\end{aligned}
\right.
\end{equation}
Similar to the proof of Theorem \ref{thm-e}, we know that $x_h(t)$ is $\mathcal{F}_t$-measurable and $\mathbb{E}\int_0^T(x_h(t))^2dt <\infty$, which implies that $x_h(t) \in L^2_{ad}([0,T]\times \Omega,\mathbb{R}^n)$.
Define a mapping
$
\tilde{F}_i:L^2_{ad}([t_i,t_{i+1}]\times \Omega,\mathbb{R}^n) \times U[t_i,t_{i+1}] \to L^2_{ad}([t_i,t_{i+1}]\times \Omega,\mathbb{R}^m)
$
by
$$
	\tilde{F}_i(x,u)(s,\omega):=F(s,\omega,x(s,\omega),u(s,\omega))
$$
for all $(x,u)\in L^2_{ad}([t_i,t_{i+1}]\times \Omega,\mathbb{R}^n) \times U[t_i,t_{i+1}]$, for all $s\in[t_i,t_{i+1})$, and for all $\omega \in \Omega$. Then, by \eqref{S1-e} Lemmas \ref{lm-p} and \ref{lm-h-e}, we have
\begin{align}\label{uh-p}
u_h(\hat{s})=P_{U[t_i,t_{i+1}]}[u_h(\hat{s})-\rho \tilde{F}_n(x_h(\hat{s}), u_h(\hat{s}))],
\end{align}
where $0 < \rho < \frac{2C}{L^2_F}$ is a constant. Thus, \eqref{uh-p} and \eqref{S1-e} imply that $u_h(t)$ is $\mathcal{F}_{t_i}$-measurable for all $t\in [t_i,t_{i+1})$ and $\mathbb{E}\int_0^T\|u_h(t)\|^2dt <\infty$.

Next we consider $\mathbb{E}\|x_h(t)-x(t)\|^2$. To this end, by employing \eqref{S1-e}, one has
\begin{align*}
	&x_h(t_i)=x_0+\sum_{k=1}^{i-1}f(t_{k-1},x_h(t_{k-1}),u_h(t_{k-1}))(t_k-t_{k-1})\\
	&\qquad \qquad  \mbox{}+\sum_{k=1}^{i-1}g(t_{k-1},x_h(t_{k-1}),u_h(t_{k-1}))(B_{t_k}-B_{t_{k-1}}),
\end{align*}
which implies that
\begin{align}\label{sdf-1}
\mathbb{E}\|x_h(t_i)\|^2
&\leq 3\mathbb{E}\|x_0\|^2
+3\mathbb{E}\left[\sum_{k=1}^{i-1}f(t_{k-1},x_h(t_{k-1}),u_h(t_{k-1}))(t_k-t_{k-1})\right]^2\\
&\mbox{}\quad+3\mathbb{E}\left[\sum_{k=1}^{i-1}g(t_{k-1},x_h(t_{k-1}),u_h(t_{k-1}))(B_{t_k}-B_{t_{k-1}})\right]^2\nonumber\\
&\leq 3\mathbb{E}(x_0)^2+3H_1+3H_2,\nonumber
\end{align}
where
$$H_1=h^2\mathbb{E}\left[\sum\limits_{k=1}^{i-1}f(t_{k-1},x_h(t_{k-1}),u_h(t_{k-1}))\right]^2,$$ $$H_2=h\sum\limits_{k=1}^{i-1}\mathbb{E}\|g(t_{k-1},x_h(t_{k-1}),u_h(t_{k-1}))\|_{\mathbb{R}^{n\times m}}^2.$$
By Remark \ref{lgc}, there exist two constants $K_1,K_2>0$ such that, for all $(t,\omega) \in [0,T]\times \Omega$,
$$\|f(t,x,u)\|\le K_1(1+\|x\|+\|u\|), \quad \|g(t,x,u)\|_{\mathbb{R}^{n\times m}} \le K_2(1+\|x\|+\|u\|).$$
Thus, we have
\begin{align}\label{sdf-1-f}
H_1 \leq & h^2K_1^2\mathbb{E}\left[\sum_{k=1}^{i-1}(1+\|x_h(t_{k-1})\|+\|u_h(t_{k-1})\|)\right]^2\\
\leq
&h^2K_1^2N \sum_{k=1}^{i-1}\mathbb{E}\left[1+\|x_h(t_{k-1})\|+\|u_h(t_{k-1})\|\right]^2\nonumber\\
\leq
& 3ThK_1^2 \sum_{k=1}^{i-1}\mathbb{E}[1+\|x_h(t_{k-1})\|^2+\|u_h(t_{k-1})\|^2]\nonumber\\
\leq
& 3ThK_1^2 +3ThK_1^2\sum_{k=1}^{i-1}\mathbb{E}\|x_h(t_{k-1})\|^2+3ThK_1^2\sum_{k=1}^{i-1}\mathbb{E}\|u_h(t_{k-1})\|^2\nonumber
\end{align}
and
\begin{align}\label{sdf-1-g}
H_2&\leq 3hK_2^2 \sum_{k=1}^{i-1}\mathbb{E}[1+\|x_h(t_{k-1})\|^2+\|u_h(t_{k-1})\|^2]\\
&\leq 3hK_2^2+3hK_2^2\sum_{k=1}^{i-1}\mathbb{E}\|x_h(t_{k-1})\|^2+3hK_2^2\sum_{k=1}^{i-1}\mathbb{E}\|u_h(t_{k-1})\|^2.\nonumber
\end{align}
Combining inequalities \eqref{sdf-1} to \eqref{sdf-1-g}, we obtain
\begin{align}\label{sdf-2}
\mathbb{E}\|x_h(t_i)\|^2
\leq 3\mathbb{E}\|x_0\|^2+Ch\left[1+\sum_{k=1}^{i-1}\mathbb{E}\|x_h(t_{k-1})\|^2+\sum_{k=1}^{i-1}\mathbb{E}\|u_h(t_{k-1})\|^2\right],
\end{align}
where $C=9TK_1^2+9K_2^2$. By Lemma \ref{lm-ubd-1}, there exists a constant $M>0$ such that $\mathbb{E}\|u_h(t_i)\|^2 \leq M(1+\mathbb{E}\|x_h(t_i)\|^2)$ and so
 \eqref{sdf-2} implies that
\begin{align}\label{sdf-3}
\mathbb{E}\|x_h(t_i)\|^2
\leq 3\mathbb{E}\|x_0\|^2+CT+C'h\sum_{k=1}^{i-1}\mathbb{E}\|x_h(t_{k-1})\|^2,
\end{align}
where $C'=CM+C$. Clearly, $h\leq \frac{T}{N^{\frac{1}{2}}i^{\frac{1}{2}}}$.  Now applying \eqref{sdf-3} and Lemma \ref{gGi}, one has
\begin{align}\label{x_h-c}
\mathbb{E}\|x_h(t_i)\|^2
\leq 3\mathbb{E}\|x_0\|^2+CT+\frac{C'T}{N^{\frac{1}{2}}}\sum_{k=1}^{i-1}\frac{\mathbb{E}\|x_h(t_{k-1})\|^2}{(i-k)^{\frac{1}{2}}}
\leq A,
\end{align}
where $A=(3\mathbb{E}(x_0)^2+CT)[1+\mathbb{G}_{\frac{1}{2}}(C'T\Gamma(\frac{1}{2}))]$.  Moreover, for any fixed $t\in [t_{i-1},t_i)$, it follows from \eqref{x_h-c} and Assumption \ref{a-e} that
\begin{align}\label{x-h-i}
&\mathbb{E}\|x_h(t)-x_h(t_{i-1})\|^2\\
=&\mathbb{E}\left[\int_{t_{i-1}}^tf(t_{i-1},x_h(t_{i-1}),u_h(t_{i-1}))ds+\int_{t_{i-1}}^tg(t_{i-1},x_h(t_{i-1}),u_h(t_{i-1}))dB_s\right]^2\nonumber\\
\leq& 2\mathbb{E}\left[\int_{t_{i-1}}^tf(t_{i-1},x_h(t_{i-1}),u_h(t_{i-1}))ds\right]^2\nonumber\\
&+2\mathbb{E}\left[\int_{t_{i-1}}^tg(t_{i-1},x_h(t_{i-1}),u_h(t_{i-1}))dB_s\right]^2\nonumber\\
=&2\mathbb{E}[\|f(t_{i-1},x_h(t_{i-1}),u_h(t_{i-1})\|^2(t-t_{i-1})^2]\nonumber\\
&+2\mathbb{E}[\|g(t_{i-1},x_h(t_{i-1}),u_h(t_{i-1})\|_{\mathbb{R}^{n\times m}}^2(t-t_{i-1})]\nonumber\\
\leq
&\mathbb{E}[6K_1^2(1+\|x_h(t_{i-1})\|^2+\|u_h(t_{i-1})\|^2)(t-t_{i-1})^2]\nonumber\\
&+\mathbb{E}[6K_2^2(1+\|x_h(t_{i-1})\|^2+\|u_h(t_{i-1})\|^2)(t-t_{i-1})]\nonumber\\
\leq & B_1h^2+B_2h,\nonumber
 \end{align}
where $B_1=6K_1^2(1+A+M+AM)$ and $B_2=6K_2^2(1+A+M+AM)$. Thus, by \eqref{sdvi-h-s^}, we have
\begin{align}\label{x_h-x_t}
\mathbb{E}\|x_h(t)-x(t)\|^2 \leq I_1+I_2,
\end{align}
where
\begin{equation}
\left\{
\begin{aligned}
&I_1=2\mathbb{E}\left[\int_0^t(f(s,x(s),u(s))-f(\hat{s},x_h(\hat{s}),u_h(\hat{s})))ds\right]^2,\nonumber\\
&I_2=2\mathbb{E}\left[\int_0^t(g(s,x(s),u(s))-g(\hat{s},x_h(\hat{s}),u_h(\hat{s})))dB_s\right]^2.\nonumber
\end{aligned}
\right.
\end{equation}
From \eqref{x-h-i}, Cauchy-Schwartz's inequality and Assumption \ref{a-e}, one has
\begin{align}\label{I1}
I_1&=2\mathbb{E}\left[\int_0^t(f(s,x(s),u(s))-f(\hat{s},x_h(\hat{s}),u_h(\hat{s})))ds\right]^2\\
&\leq 2T\mathbb{E}\left[\int_0^t\|f(s,x(s),u(s))-f(\hat{s},x_h(\hat{s}),u_h(\hat{s}))\|^2ds\right]\nonumber\\
&\leq 2T L_f^2 \mathbb{E}\left[\int_0^t|s-\hat{s}|+\|x(s)-x_h(\hat{s})\|+\|u(s)-u_h(\hat{s})\|\right]^2ds\nonumber\\
&\leq 6T L_f^2 \mathbb{E}\left[\int_0^t|s-\hat{s}|^2+\|x(s)-x_h(\hat{s})\|^2+\|u(s)-u_h(\hat{s})\|^2\right]ds\nonumber\\
&\leq 6T^2 L_f^2h^2+ 12T L_f^2 \mathbb{E}\left[\int_0^t\|x(s)-x_h(s)\|^2ds\right]\nonumber\\
&\mbox{}\quad +12T L_f^2 \mathbb{E}\left[\int_0^t\|x_h(s)-x_h(\hat{s})\|^2ds\right]+6T L_f^2\mathbb{E}\left[\int_0^t\|u(s)-u_h(\hat{s})\|^2ds\right]\nonumber
\end{align}
\begin{align*}
&\leq 6T^2 L_f^2h^2+12T^2 L_f^2B_1h^2+12T^2 L_f^2B_2h\nonumber\\
&\mbox \quad +6T L_f^2\mathbb{E}\left[\int_0^t\|u(s)-u_h(\hat{s})\|^2ds\right]++12T L_f^2 \mathbb{E}\left[\int_0^t\|x_h(s)-x_h(\hat{s})\|^2ds\right]\nonumber
\end{align*}
and
\begin{align}\label{I2}
I_2&=2\mathbb{E}\left[\int_0^t(g(s,x(s),u(s))-g(\hat{s},x_h(\hat{s}),u_h(\hat{s})))dB_s\right]^2\\
&= 2\mathbb{E}\left[\int_0^t\|g(s,x(s),u(s))-g(\hat{s},x_h(\hat{s}),u_h(\hat{s}))\|^2_{\mathbb{R}^{n\times m}}ds \right]\nonumber\\
&\leq 2L_g^2 \mathbb{E}\int_0^t\left[|s-\hat{s}|+\|x(s)-x_h(\hat{s})\|+\|u(s)-u_h(\hat{s})\|\right]^2ds\nonumber\\
&\leq 6L_g^2 \mathbb{E}\left[\int_0^t|s-\hat{s}|^2+\|x(s)-x_h(\hat{s})\|^2+\|u(s)-u_h(\hat{s})\|^2\right]ds\nonumber\\
&\leq 6TL_g^2h^2+12L_g^2\mathbb{E}\left[\int_0^t\|x(s)-x_h(s)\|^2ds\right]\nonumber\\
&\mbox \quad+12L_g^2\mathbb{E}\left[\int_0^t\|x_h(s)-x_h(\hat{s})\|^2ds\right]+6L_g^2\mathbb{E}\left[\int_0^t\|u(s)-u_h(\hat{s})\|^2ds\right]\nonumber\\
&\leq 6TL_g^2h^2+12TL_g^2B_1^2h^2+12TL_g^2B_2h+12L_g^2\mathbb{E}\left[\int_0^t\|x(s)-x_h(s)\|^2ds\right]\nonumber\\
&\mbox \quad+6L_g^2\mathbb{E}\left[\int_0^t\|u(s)-u_h(\hat{s})\|^2ds\right].\nonumber
\end{align}
Combining \eqref{x_h-x_t} to \eqref{I2}, we obtain
\begin{align}\label{x_h-x_t-1}
\mathbb{E}\|x_h(t)-x(t)\|^2 &\leq C_1h^2+C_2h+C_3\mathbb{E}\left[\int_0^t\|x(s)-x_h(s)\|^2ds\right] \\
&\quad \;+C_4\mathbb{E}\left[\int_0^t\|u(s)-u_h(\hat{s})\|^2ds\right],\nonumber
\end{align}
where $C_1=6T^2 L_f^2+6TL_g^2+12T^2 L_f^2B_1+12TL_g^2B_1^2$, $C_2=12T^2 L_f^2B_2+12TL_g^2B_2$, $C_3=12T L_f^2+12L_g^2$, and $C_4=6T L_f^2+6L_g^2$.

Finally, we estimate $\mathbb{E}\left[\int_0^t\|u(s)-u_h(\hat{s})\|^2ds\right]$. To do this, it follows from \eqref{uh-p} and Assumption \ref{a-e} that
\begin{align}\label{u_h-1}
&\mbox{}\quad \;\mathbb{E}\int_{t_{i-1}}^t \|u(s)-u(\hat{s})+\rho F(s,\omega,x(s),u_h(\hat{s}))-\rho F(s,\omega,x(s),u(s))\|^2ds\\
&\leq \mathbb{E}\int_{t_{i-1}}^t\|u(s)-u(\hat{s})\|^2ds\nonumber\\
&\mbox{}\quad+\rho^2\mathbb{E}\int_{t_{i-1}}^t\|F(s,\omega,x(s),u_h(\hat{s}))- F(s,\omega,x(s),u(s))\|^2ds\nonumber\\
&\mbox{}\quad-2\rho\mathbb{E}\left[\int_{t_{i-1}}^{t} \langle u(s)-u_h(\hat{s}),F(s,\omega,x(s),u(s))-F(s,\omega,x(s),u_h(\hat{s}))\rangle ds\right]\nonumber\\
&\leq (1-2\rho C+\rho^2L^2_F)\mathbb{E}\int_{t_{i-1}}^t\|u(t)-u_h(\hat{s})\|^2ds.\nonumber
\end{align}
Applying \eqref{u_h-1} and Assumption \ref{a-e}, we get
\begin{align*}
&\mbox{}\quad \;
\left[\mathbb{E}\int_{t_{i-1}}^t \|u(s)-u_h(\hat{s})\|^2ds\right]^\frac{1}{2}\nonumber\\
&=\Bigg[\mathbb{E}\int_{t_{i-1}}^t \|P_{U[t_{n-1},t]}(u-\rho \tilde{F}_n(x,u))\nonumber\\
&\mbox{}\quad-P_{U[t_{i-1},t]}(u(\hat{s})-\rho \tilde{F}_n(x_h(\hat{s}),u_h(\hat{s})))\|^2ds\Bigg]^{\frac{1}{2}}\nonumber\\
&\leq \Bigg[\mathbb{E}\int_{t_{i-1}}^t \|u(s)-u_h(\hat{s})+\rho F(\hat{s},\omega,x_h(\hat{s}),u_h(\hat{s}))-\rho F(s,\omega,x(s),u(s))\|^2ds\Bigg]^{\frac{1}{2}}\nonumber\\
&=\Bigg[\mathbb{E}\int_{t_{i-1}}^t\|u(s)-u_h(\hat{s})+\rho F(\hat{s},\omega,x_h(\hat{s}),u_h(\hat{s}))-\rho F(s,\omega,x(s),u_h(\hat{s}))\nonumber\\
&\mbox{}\quad +\rho F(s,\omega,x(s),u_h(\hat{s}))-\rho F(s,\omega,x(s),u(s))\|^2ds\Bigg]^{\frac{1}{2}}\nonumber\\
&\leq \left[\mathbb{E}\int_{t_{i-1}}^t\|u(s)-u_h(\hat{s})+\rho F(s,\omega,x(s),u_h(\hat{s}))-\rho F(s,\omega,x(s),u(s))\|^2ds\right]^{\frac{1}{2}}\nonumber\\
&\mbox{}\quad+\left[\mathbb{E}\int_{t_{i-1}}^t\|\rho F(\hat{s},\omega,x_h(\hat{s}),u_h(\hat{s}))-\rho F(s,\omega,x(s),u_h(\hat{s}))\|^2ds\right]^{\frac{1}{2}}
\end{align*}
\begin{align*}
&\leq \sqrt{(1-2\rho C+\rho^2L^2_F)}\mathbb{E}\left[\int_{t_{i-1}}^{t}\|u(s)-u_h(\hat{s})\|^2ds\right]^{\frac{1}{2}}+\sqrt{2}\rho L_F h^{\frac{3}{2}}\nonumber\\
&\mbox{}\quad + \sqrt{2}\rho L_F \left[\mathbb{E}\int_{t_{i-1}}^t\|x(s)-x_h(\hat{s})\|^2ds\right]^{\frac{1}{2}}
\end{align*}
and so
\begin{align*}
&\mathbb{E}\left[\int_{t_{i-1}}^{t}\|u(s)-u_h(\hat{s})\|^2ds\right]^{\frac{1}{2}}\leq A_F \mathbb{E}\left[\int_{t_{i-1}}^{t}\|x_h(\hat{s})-x(s)\|^2ds\right]^{\frac{1}{2}}+A_F h^{\frac{3}{2}},
\end{align*}
which implies
\begin{align*}
&\mathbb{E}\int_{t_{i-1}}^{t}\|u(s)-u_h(\hat{s})\|^2ds\leq 2A^2_F \mathbb{E}\int_{t_{i-1}}^{t}\|x_h(\hat{s})-x(s)\|^2ds+2A^2_F h^{3}.
\end{align*}
Then,
\begin{align}\label{u_t-u_h-1}
\mathbb{E}\int_0^t \|u(s)-u_h(\hat{s})\|^2ds&=\mathbb{E}\sum_{k=0}^{i-1}\int_{t_k}^{t_{k+1}}\|u(s)-u_h(\hat{s})\|^2ds \\
&\quad \;+\mathbb{E}\int_{t_{i-1}}^t\|u(s)-u_h(\hat{s})\|^2ds \nonumber\\
&\leq 2A^2_F \mathbb{E}\sum_{k=0}^{i-1}\int_{t_k}^{t_{k+1}}\|x_h(\hat{s})-x(s)\|^2ds\nonumber \\
&\quad \;+2A^2_F\mathbb{E}\int_{t_{i-1}}^t\|x_h(\hat{s})-x(s)\|^2ds +2iA^2_F h^{3} \nonumber\\
&\leq 2A^2_F \mathbb{E}\int_{0}^{t}\|x_h(\hat{s})-x(s)\|^2ds+2iA^2_F h^{3},\nonumber
\end{align}
where $A_F=\frac{\sqrt{2}\rho L_F}{1-\sqrt{(1-2\rho C+\rho^2L^2_F)}}$.
It follows from \eqref{x-h-i} and \eqref{u_t-u_h-1} that
\begin{align}\label{u_t-u_h}
&\mbox{}\quad \mathbb{E}\left[\int_0^t \|u(s)-u_h(\hat{s})\|^2ds\right]\\
&\leq 2A^2_F \mathbb{E}\left[\int_{0}^{t}\|x_h(\hat{s})-x(s)\|^2ds\right]+2iA^2_F h^{3}\nonumber\\
&\leq 2A^2_F \mathbb{E}\left[\int_{0}^{t}\|x_h(\hat{s})-x_h(s)+x_h(s)-x(s)\|^2ds\right]+2iA^2_F h^{3}\nonumber\\
&\leq 4A^2_FTB_1h^2+4A^2_FTB_2h+4A^2_F \mathbb{E}\left[\int_0^t\|x_h(s)-x(s)\|^2ds\right]+2iA^2_F h^{3}\nonumber
\end{align}
Combining \eqref{x_h-x_t-1} and \eqref{u_t-u_h}, one has
\begin{align}\label{xh-x-gi}
&\mathbb{E}\|x_h(t)-x(t)\|^2\leq C'_1h^2+C'_2h+C'_3\mathbb{E}\left[\int_0^t\|x(s)-x_h(s)\|^2ds\right],
\end{align}
where $C'_1=C_1+4C_4A^2_FTB_1+2A^2_FT$, $C'_2=C_2+4C_4A^2_FTB_2$, and $C'_3=C_3+4C_4A^2_F$.  Applying Gronwall's inequality in \eqref{xh-x-gi}, we get
\begin{align}\label{x_h-x_t-2}
&\mathbb{E}\|x_h(t)-x(t)\|^2\leq (C'_1h^2+C'_2h)+(C'_1h^2+C'_2h)(e^{C'_3t}-1).
\end{align}
Thus, it follows from \eqref{x_h-x_t-2} and \eqref{u_t-u_h} that
\begin{align*}
\|x_h-x\|^2_{H_{[0,T]}}&\leq (C'_1h^2+C'_2h)T+(C'_1h^2+C'_2h)T(e^{C'_3T}-1)\\
&\leq (C'_1T+C'_1T(e^{C'_3T}-1)) h^2 +(C'_2T+C'_2T(e^{C'_3T}-1))h
\end{align*}
and
\begin{eqnarray*}
&&\|u-u_h\|^2_{H_{[0,T]}}\\
&=&\mathbb{E}\int_0^T\|u(s)-u_h(s)\|^2ds\\
&=&\mathbb{E}\int_0^T\|u(s)-u_h(\hat{s})\|^2ds\\
&\leq& (4A^2_FTB_1+2A^2_FT)h^2+4A^2_FTB_2h +4A^2_F \|x_h(s)-x(s)\|_{H_{[0,T]}}^2\\
&\leq& (4A^2_FB_1+2A^2_F+4A^2_FC_1')h^2T+(4A^2_FB_2+4A^2_FC_2')Th\\
&&\;+4A^2_F(C'_1h^2+C'_2h)T(e^{C'_3T}-1)\\
&\leq&
(4A^2_FTB_1+2A^2_FT+4A^2_FC'_1T+4A^2_FC'_1T(e^{C'_3T}-1))h^2\\
&&\;+(4A^2_FTB_2+4A^2_FC'_2T+4A^2_FC'_2T(e^{C'_3T}-1))h.
\end{eqnarray*}
Thus, we know that there exists a constant $C>0$ such that
\begin{align*}
	\|x_h-x\|^2_{H_{[0,T]}} \leq C[h^2+h], \quad
	\|u-u_h\|_{H_{[0,T]}}^2\leq C[h^2+h].
\end{align*}
Clearly, $\|x_h-x\|_{H_{[0,T]}} \rightarrow 0$ and  $\|x_h-x\|_{H_{[0,T]}} \rightarrow 0$ as $h\rightarrow 0$.
\end{proof}

\section{Two applications}
In this section, we first propose Algorithm \ref{al} for solving SDVI \eqref{S1} and then give two applications of Algorithm \ref{al} to solve the electrical circuits with diodes and the collapse of the bridge problems in stochastic environment.

From \cite{Glasserman2003Monte}, by Monte Carlo method, we consider the following approximation of a standard $l$-dimensional Brownian motion $B_t$:
\begin{align}\label{aBm}
B_{t_{k+1}}-B_{t_{k}}=\sqrt{t_{k+1}-t_k}Z_{t_{k+1}},\quad k=0,1,\cdots,N-1,
\end{align}
where $\{Z_{t_{k+1}}\}_{k=0}^{N-1}$ are independent $l$-dimensional standard normal random variables. It is well known that
\begin{align*}
&\mathbb{E}(B_{t_{k+1}}-B_{t_{k}})=\mathbb{E}(\sqrt{t_{k+1}-t_k}Z_{t_{k+1}})=0,\\
&\mathbb{E}(B_{t_{k+1}}-B_{t_{k}})^2=\mathbb{E}(\sqrt{t_{k+1}-t_k}Z_{t_{k+1}})^2=t_{k+1}-t_k.
\end{align*}

Now we propose the following algorithm for solving SDVI \eqref{S1}.

\begin{Al}\label{al} \mbox{}

\begin{itemize}
\item[{\bf Step 0}:] First divide the time $[0,T]$ $(h=\frac{T}{N})$ into $N$ intervals,
\begin{align*}
t_{0}=0< t_{1}=h < \cdots < t_{N}=T.
\end{align*}

\item[{\bf Step 1}:] Let $x_h(t_{0})=x_h(0)$ and $u_h(t_{0})=u_h(0)\in K$.

\item[{\bf Step 2}:] For $k=0,1,2,\cdots,N-1$, let $Z(t_{k+1})$ be a standard normal random variable obeying $N(0,t_{k+1})$ and
\begin{align*}
x_h(t_{k+1})=x_h(t_k)+h\left[f(t_{k},x_h(t_{k}),u_h(t_{k}))\right]+\left[g(t_{k},x_h(t_{k}),u_h(t_{k}))\right]\sqrt{h}Z_{t_{k+1}}.
\end{align*}
Compute $u_h(t_{k+1})\in K$ satisfying the following stochastic variational inequality problem:
\begin{align*}
\left\langle F(t_{k+1},\omega,x_h(t_{k+1}),u_h(t_{k+1})),v-u_h(t_{k+1}) \right\rangle \geq 0, \quad \forall v\in K, \; a.s. \; \omega\in \Omega.
\end{align*}

\item[{\bf Step 3}:] By the recursion, for $k=0,1,2,\cdots,N-1$, we link the points of $x_h(t_k)$ when $t\in [t_k, t_{k+1})$ by the following procedure
\begin{align*}
x_h(t)=&x_h(t_k)+\left[f(t_{k},x_h(t_{k}),u_h(t_{k}))\right](t-t_k)\\
&+\left[g(t_{k},x_h(t_{k}),u_h(t_{k}))\right]\sqrt{t-t_k}Z_{t_{k+1}}
\end{align*}
and link the points of $u_h(t_k)$ by the piecewise constant method.
\end{itemize}
\end{Al}

Clearly, Theorem \ref{th-1} ensures the convergence of Euler scheme constructed by Algorithm \ref{al}.

\subsection{The electrical circuits with diodes}
In this subsection, we consider the electrical circuits with (ideal) diodes in the stochastic environment. Following \cite{Pregla1998Grundlagen,Han2009Convergence,Chen2013Convergence}, it is well known that DVIs can be used to model the electrical circuits with (ideal) diodes. Assume that $x_1$ is the electric current across the inductor and $x_2$ is the electric tension  through the capacitor. Let $(v_{D_{i}}, i_{D_{i}} )$ (i=1,2,3,4) denote the voltage-current pairs associated to the $i$-th diode.  Setting $x(t)=(x_1(t),x_2(t))^T$ and $u(t)=(u_1(t),u_2(t),u_3(t),u_4(t))^T=(i_{D_1},v_{D_2},v_{D_3},i_{D_4})^T$, then the electrical circuits with (ideal) diodes can be described as the following DVI (see \cite{Han2009Convergence,Chen2013Convergence})
\begin{equation}\label{ex-1+}
\left\{
\begin{aligned}
&dx(t)=[Ax(t)+Bu(t)+f(t)]dt,\quad x(0)=x_0,\\
&\left\langle Qx(t)+Mu(t),v-u(t) \right\rangle \geq 0, \quad \forall v\in K, a.e. \; t\in [0,T],
\end{aligned}
\right.
\end{equation}
where
\begin{align*}
&A=\left( \begin{matrix} -\frac{2}{3} & 0\\ 0& -\frac{1}{5} \end{matrix} \right),\quad
B=\left( \begin{matrix} 0 & \frac{1}{3} & -\frac{1}{3} & 0 \\ 1 & 0 & 0 & 1 \end{matrix} \right), \quad f(t)=(2\sin(3t-\frac{\pi}{3}),0)^T,\\
&Q=\left( \begin{matrix} 0 & 1 \\ 1 & 0 \\ -1 & 0\\ 0 & 1 \end{matrix} \right), \quad
M=\left( \begin{matrix} \epsilon & 0 & -1 & 0 \\ 0 & \epsilon & 0 & 1 \\ 1 & 0 & \epsilon & 0 \\  0 & -1 & 0 & \epsilon \end{matrix} \right)(\epsilon>0),\\
&K=\left\{y\in \mathbb{R}^n: -10\leq y_1,y_2 \leq 10,\; 0\leq y_3, y_4 \leq 20 \right\},\quad x_0=(-1,0)^T, \quad T=1.5.
\end{align*}
Here the coefficient $\epsilon$ describes the level of the strong monotonicity. By taking into account some stochastic environmental effects in the electrical circuits with (ideal) diodes, DVI \eqref{ex-1+} should be changed into the following SDVI:
\begin{equation}\label{ex-1}
\left\{
\begin{aligned}
&dx(t)=[Ax(t)+Bu(t)+f(t)]dt+[Cx(t)+Du(t)+g(t)]dB_t,\quad x(0)=x_0,\\
&\left\langle Qx(t)+Mu(t),v-u(t) \right\rangle \geq 0, \quad \forall v\in K,\; a.e. \; t\in [0,T], \; a.s. \; \omega\in \Omega,
\end{aligned}
\right.
\end{equation}
where $C=\left(\begin{matrix} a & 0\\ 0& a \end{matrix}\right)$, $D=\left( \begin{matrix} b & b & b & b \\ b & b & b & b \end{matrix} \right)$, $g(t)=(c \sin(t),0)^T$, and $B_t$ is a $1$-dimensional standard Brownian motion. Here the constants $a,b,c\in \mathbb{R}$ are the diffusion coefficients related to $x(t)$, $u(t)$, and time $t$, respectively.

It follows from Theorem \ref{thm-e} that \eqref{ex-1} admits a unique Carath\'{e}odory solution $(x(t),u(t))$ for fixed $a,b,c \in \mathbb{R}$ and $\epsilon>0$. Moreover, since \eqref{ex-1} satisfies Assumption \ref{a-e}, we can use Algorithm \ref{al} to simulate paths for \eqref{ex-1}. Especially, let $T=1.5$, $N=30(h=0.05)$, $x_{0.05}(0)=(-1,0)^T$ and $u_{0.05}(0)=(0,0,0,-1)^T$ in Algorithm \ref{al}. By employing MATLAB R2018a, we can obtain some numerical results as shown in Figures \ref{ex-1-1-12}-\ref{ex-1-2-56}.

Figures \ref{ex-1-1-12}-\ref{ex-1-1-56} indicate that the  electric current through the inductor, the electric tension across the capacitor, the  electric current through the first diode, the voltage across the second diode, the voltage across the third diode, and the  electric current through the fourth diode change with time respectively, for $\epsilon =0.001$ in different $a,b,c \in R$.

\begin{figure}[H]
\subfigure[]{
\includegraphics[scale=0.45]{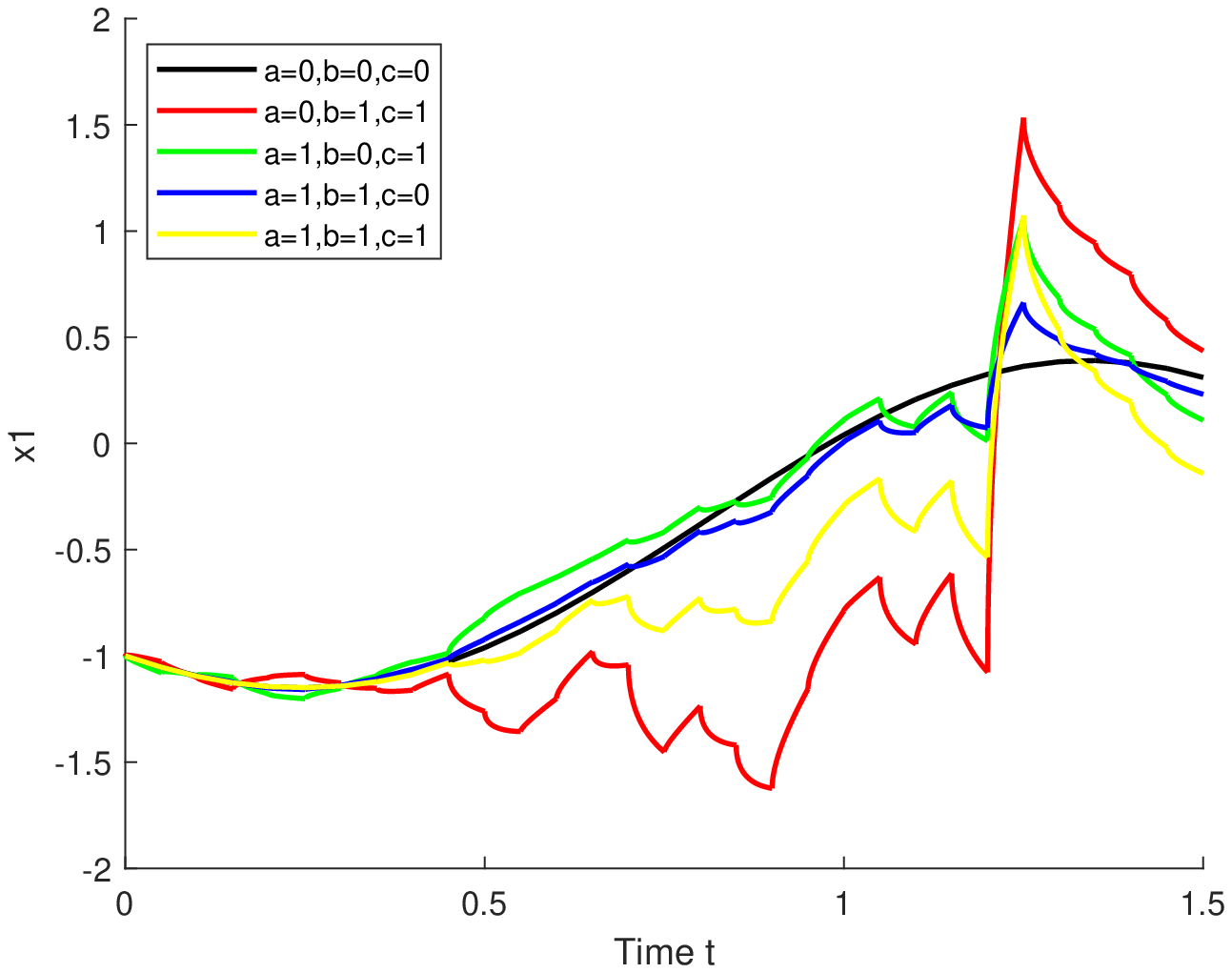}
}
\quad
\subfigure[]{
\includegraphics[scale=0.45]{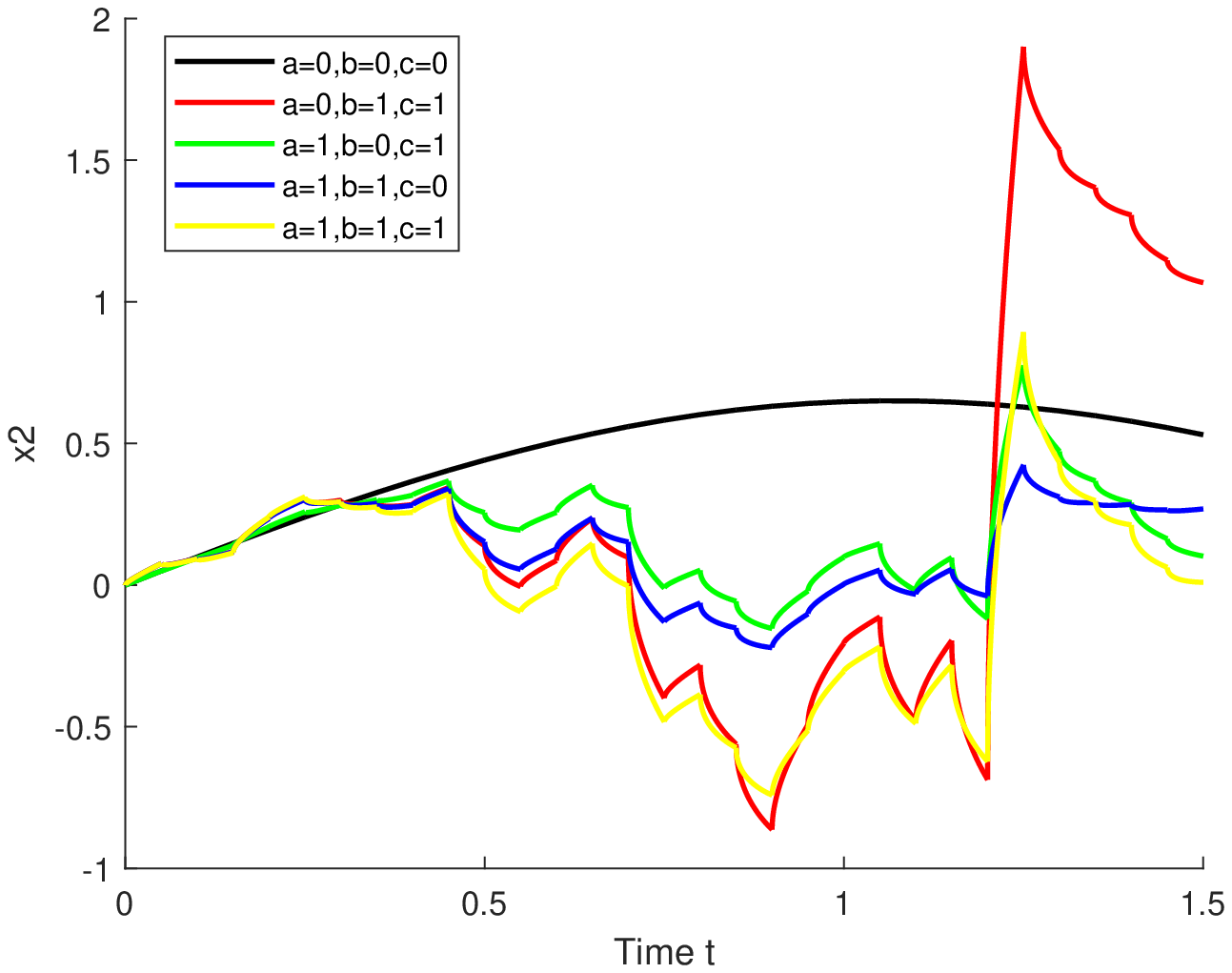}
}
\caption{The electric current through the inductor and the electric tension across the capacitor change with time in different $a,b,c \in R$.}
\label{ex-1-1-12}
\end{figure}

\begin{figure}[H]
\subfigure[]{
\includegraphics[scale=0.45]{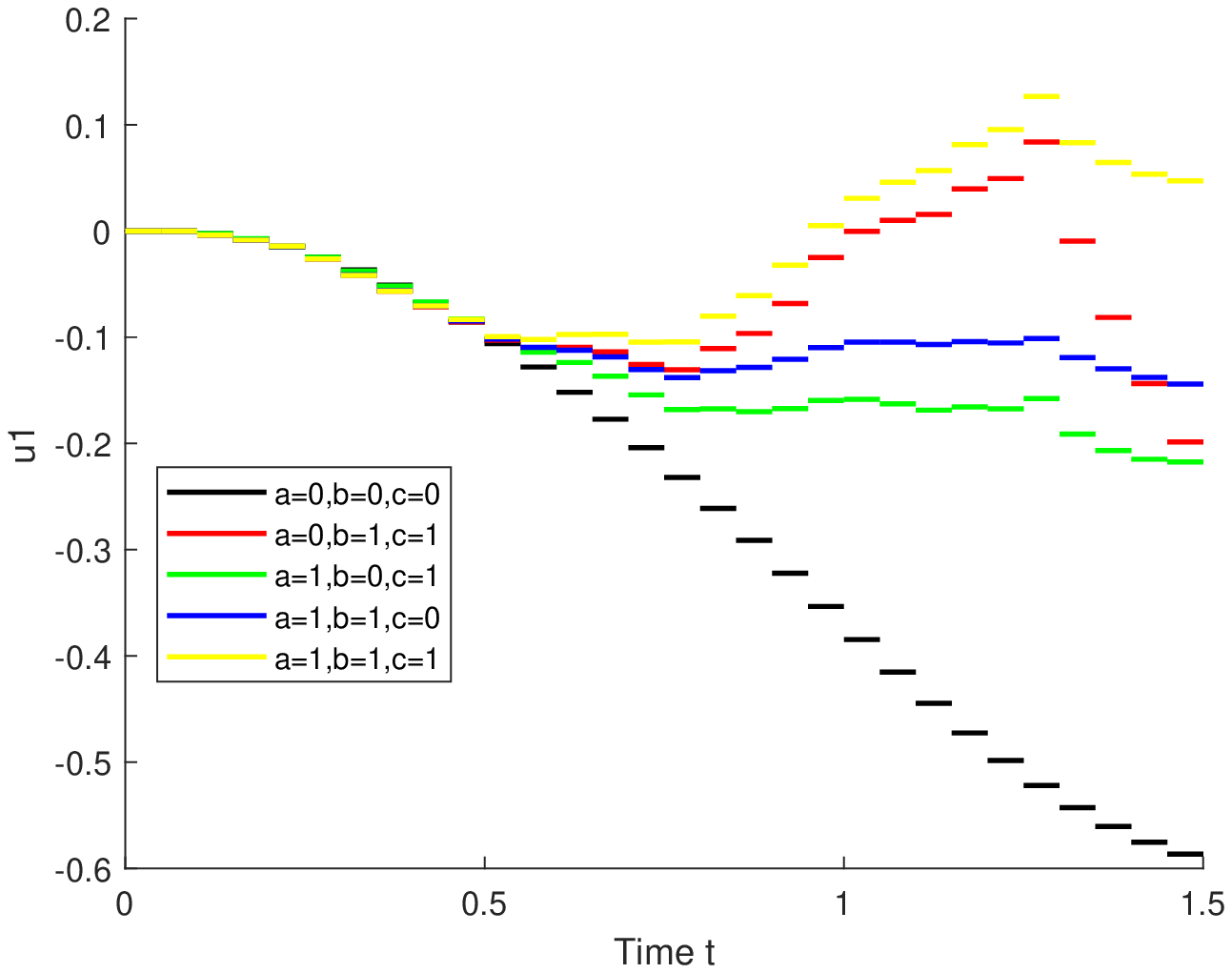}
}
\quad
\subfigure[]{
\includegraphics[scale=0.45]{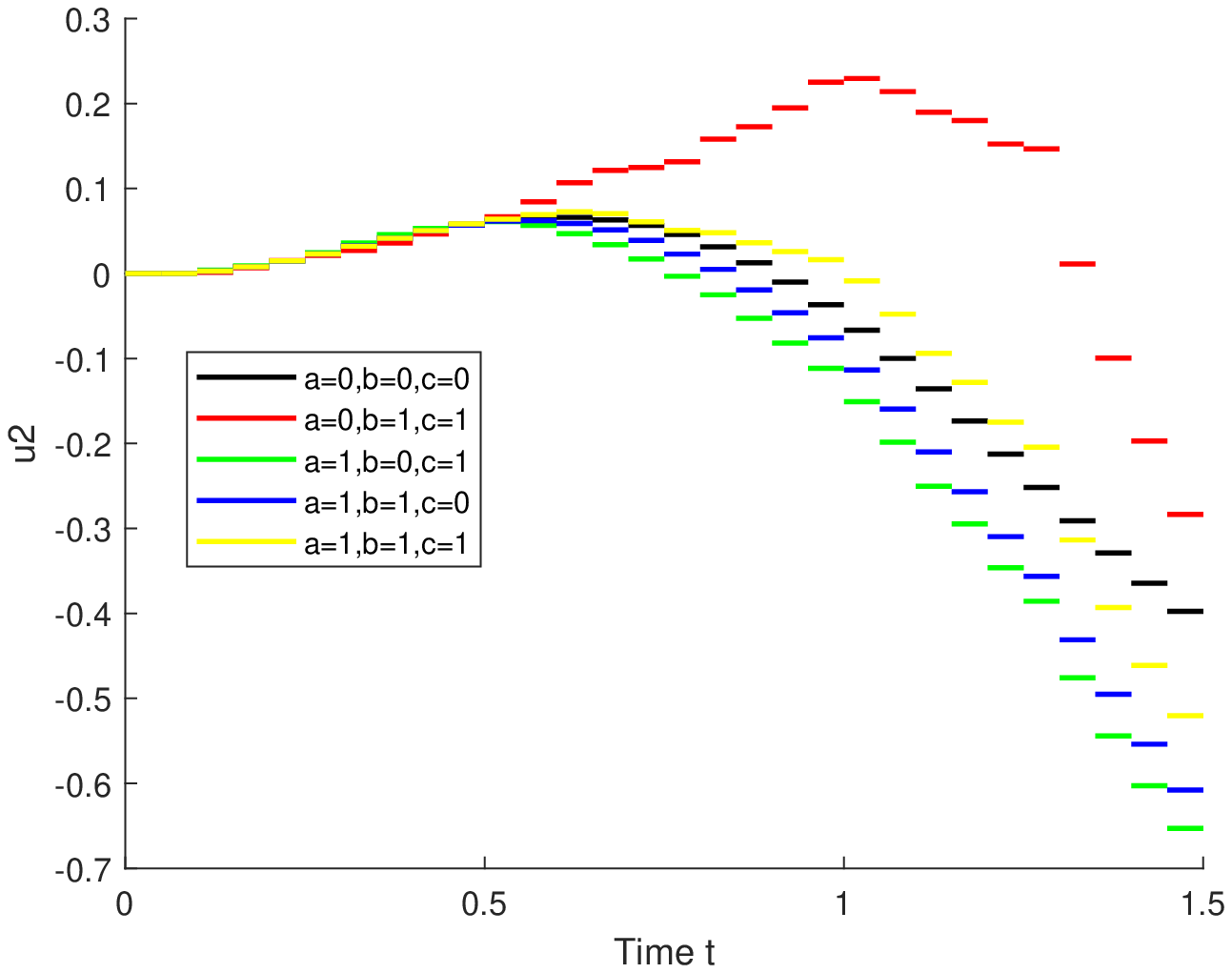}
}
\caption{The electric current across the first diode and the voltage through the second diode change with time in different $a,b,c \in R$. }
\label{ex-1-1-34}
\end{figure}

\begin{figure}[H]
\subfigure[]{
\includegraphics[scale=0.45]{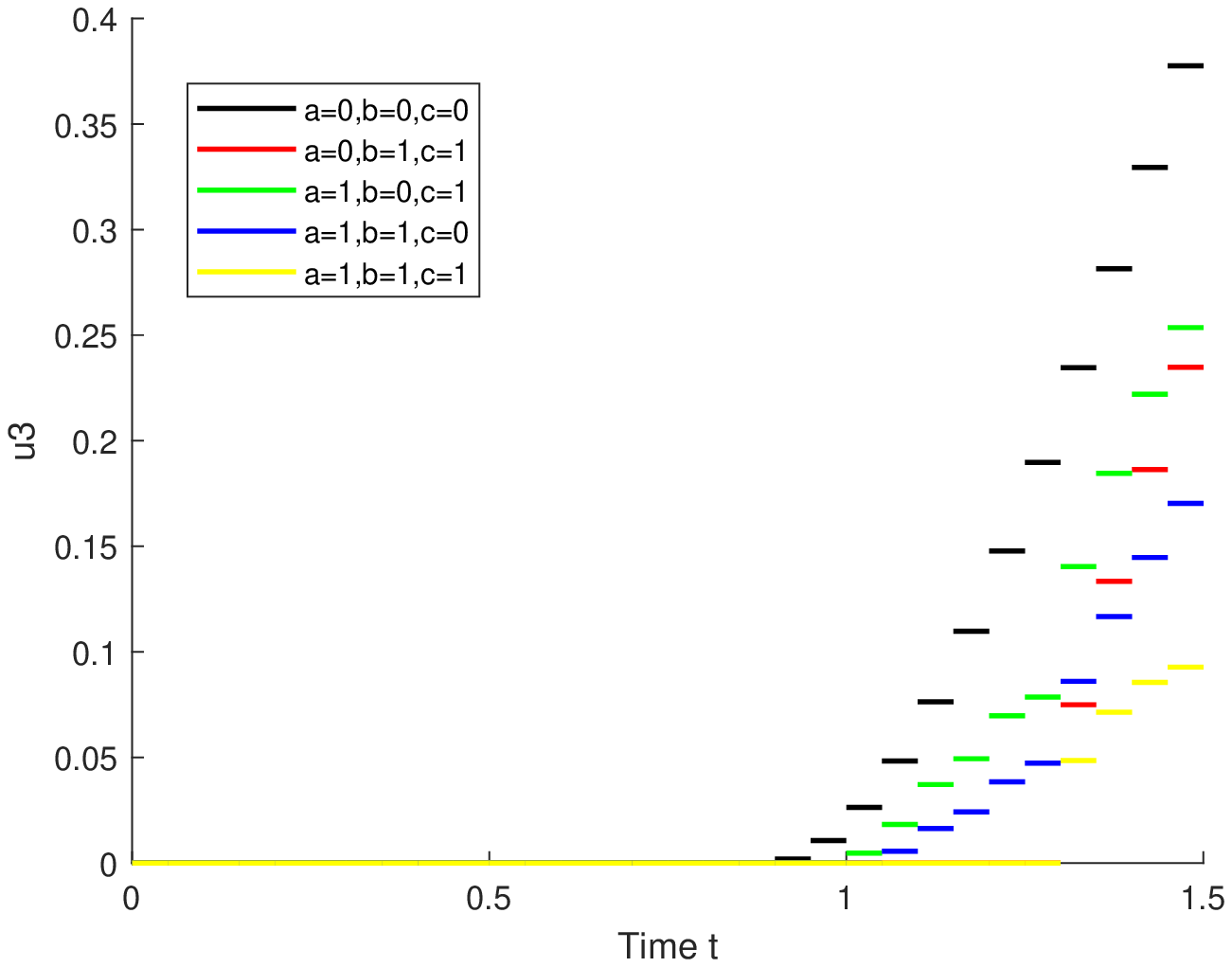}
}
\quad
\subfigure[]{
\includegraphics[scale=0.45]{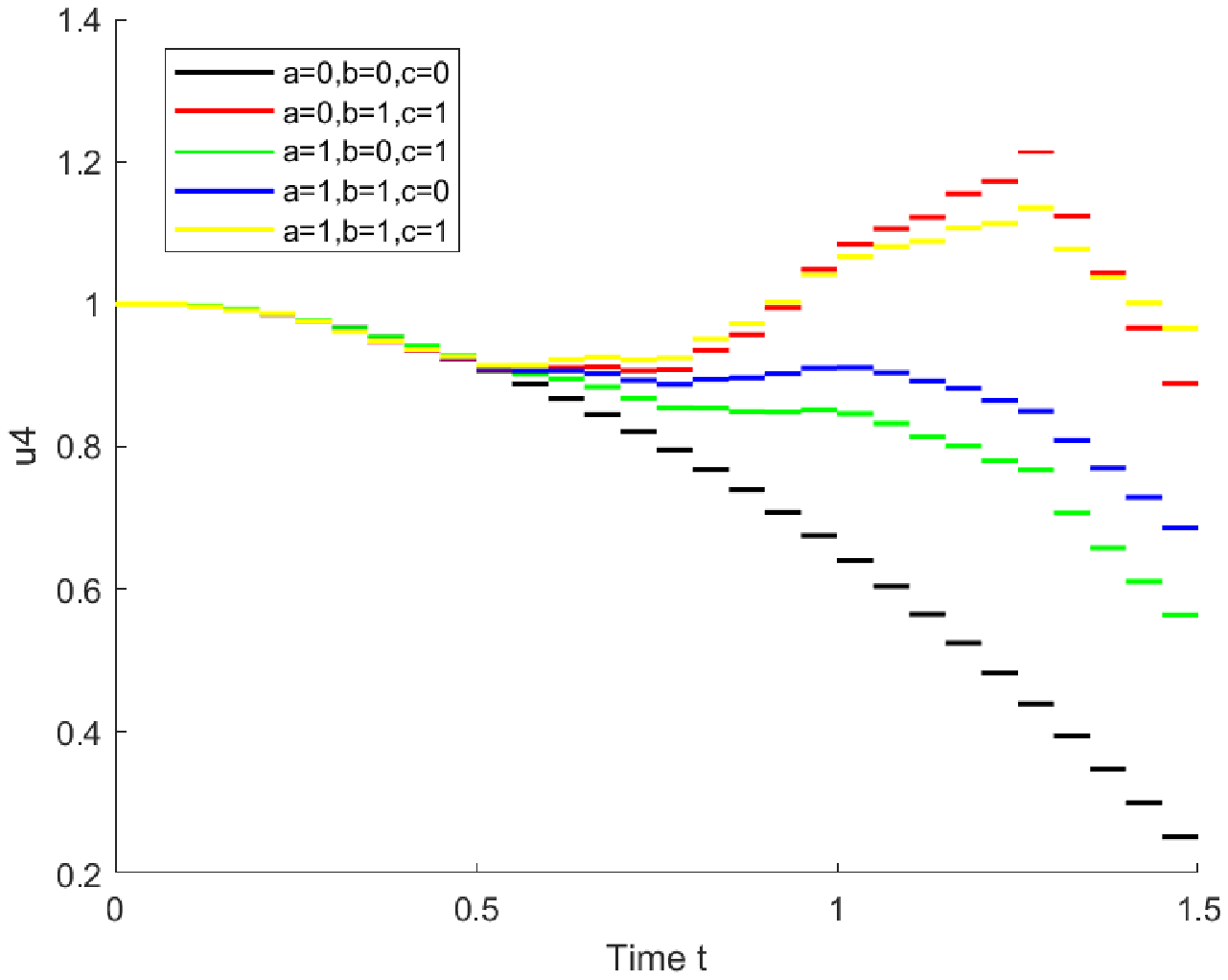}
}
\caption{The voltage across the third diode and the electric current through the fourth diode change with time in different $a,b,c \in R$.}
\label{ex-1-1-56}
\end{figure}

Figures \ref{ex-1-1-12}-\ref{ex-1-1-56} show the following facts: (i) the curves of SDVI \eqref{ex-1} and DVI are coincident when $a=0$, $b=0$, $c=0$, which is drawn by the black line; (ii) the diffusion coefficients have a great influence on the solution of SDVI \eqref{ex-1} as time going on. Moreover, Figure \ref{ex-1-1-12} tells us that the random effect of $x$ has few effects in time $[0,1.2]$ and has great influence in time $(1.2,1.5]$ on the voltage of the capacitor, and the random effects of $x$, $u$ and time $t$ have few effects on the current of inductor and the voltage of capacitor in time $[0,0.5]$; Figures \ref{ex-1-1-34} and \ref{ex-1-1-56} indicate that the random effects of $u$ has the largest influence on $i_{D_1}$, $v_{D_3}$ and $i_{D_4}$ and the random effects of $x$ has the largest effect on $v_{D_2}$.

On the other hand,  Figures \ref{ex-1-2-12}-\ref{ex-1-2-56} display that the electric current across the inductor, the electric tension through the capacitor, the electric current through the first diode, the electric tension across the second diode, the electric tension across the third diode, and the electric current through the fourth diode change with time respectively, for $a=b=c=1$ in different $\epsilon \geq 0$. Moreover, Figures \ref{ex-1-2-12}, \ref{ex-1-2-34}(a), and \ref{ex-1-2-56}(a) depict that $\epsilon$ has little influence on the current and voltage across the capacitor;  Figures \ref{ex-1-2-34}(b) and \ref{ex-1-2-56}(b) show that $\epsilon$ has great influence on the voltage. Clearly, matrix $M$ in SDVI \eqref{ex-1} does not satisfy the strong monotonicity condition in Assumption \ref{a-e} when $\epsilon=0$. Nevertheless, Figures \ref{ex-1-2-12} to \ref{ex-1-2-56} illustrate that the solution of SDVI \eqref{ex-1} when $\epsilon>0$ can approach the solution of SDVI \eqref{ex-1} when $\epsilon=0$ by taking $\epsilon$ small enough.

\begin{figure}[H]
\centering
\subfigure[]{
\includegraphics[scale=0.45]{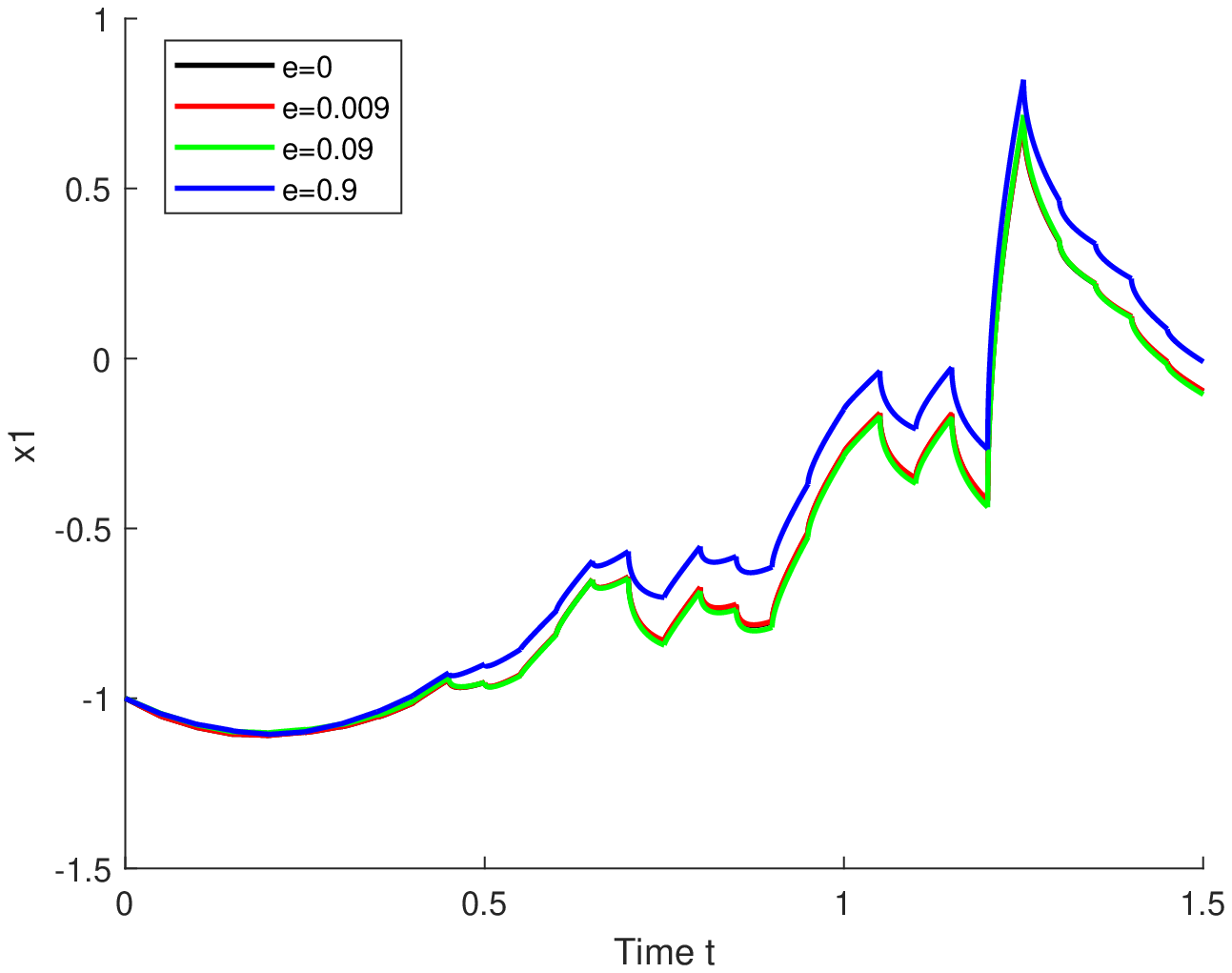}
}
\quad
\subfigure[]{
\includegraphics[scale=0.45]{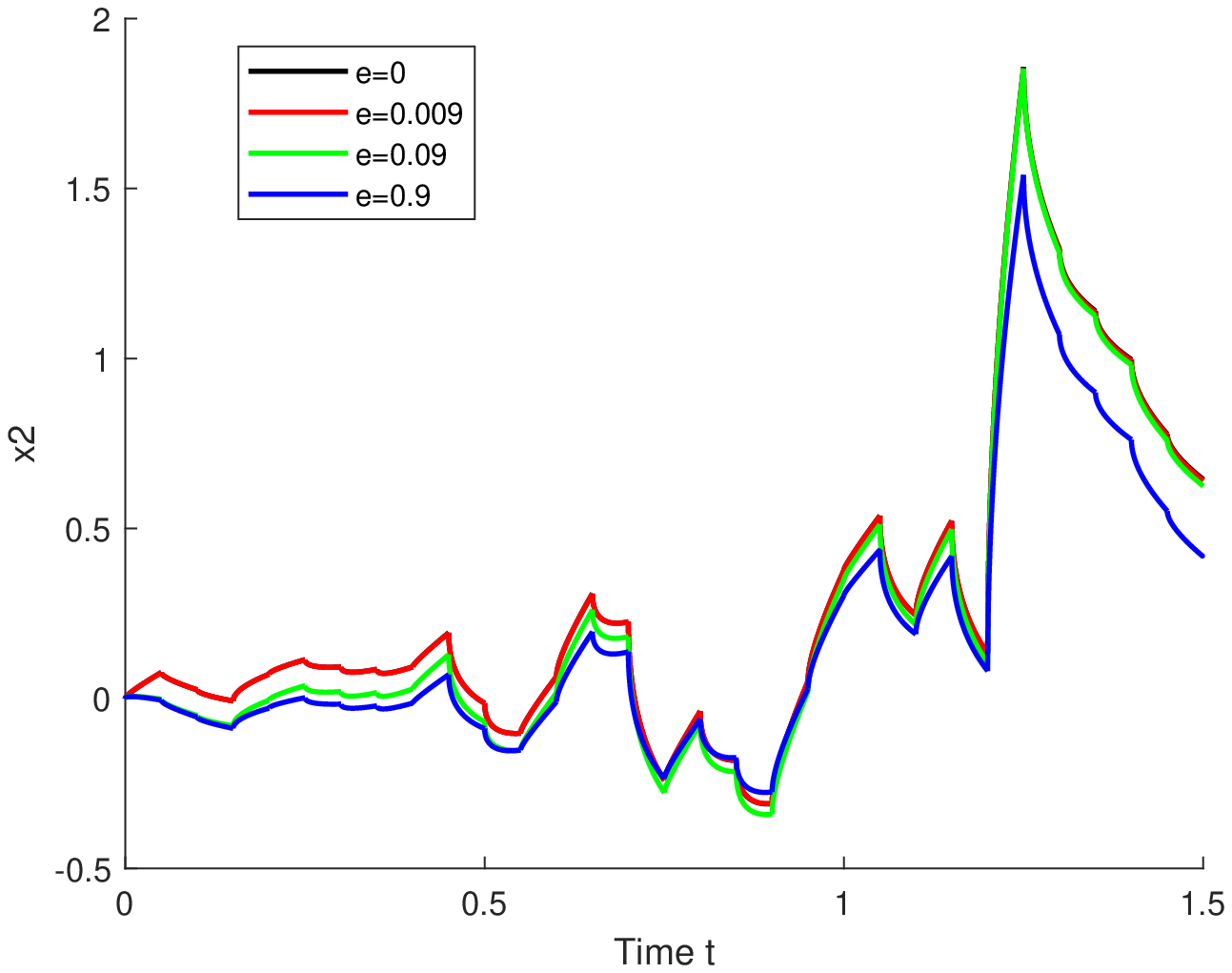}
}
\caption{The electric current across the inductor and the electric tension through the capacitor $C$ change with time in different $\epsilon$.}
\label{ex-1-2-12}
\end{figure}

\begin{figure}[H]
\subfigure[]{
\includegraphics[scale=0.45]{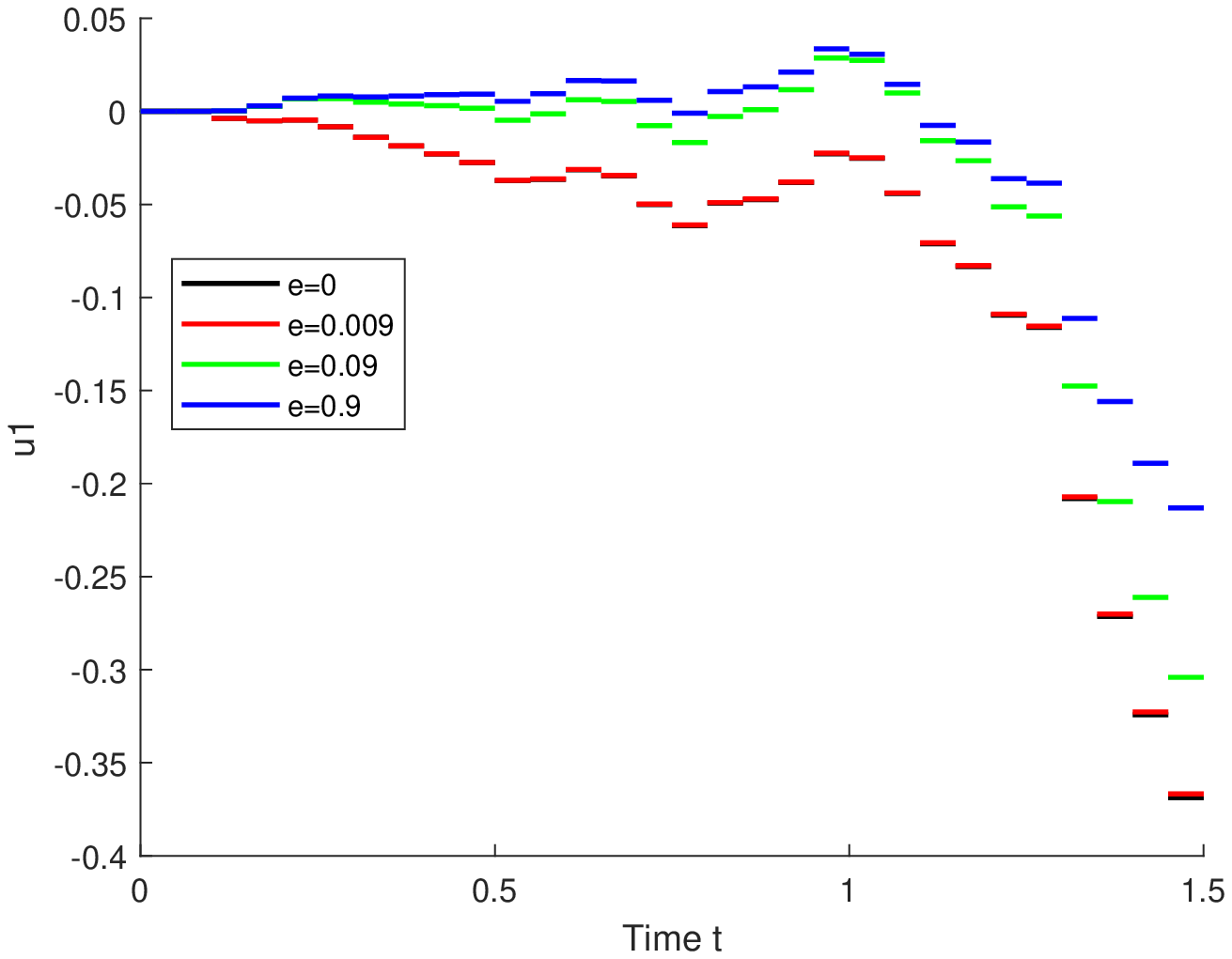}
}
\quad
\subfigure[]{
\includegraphics[scale=0.45]{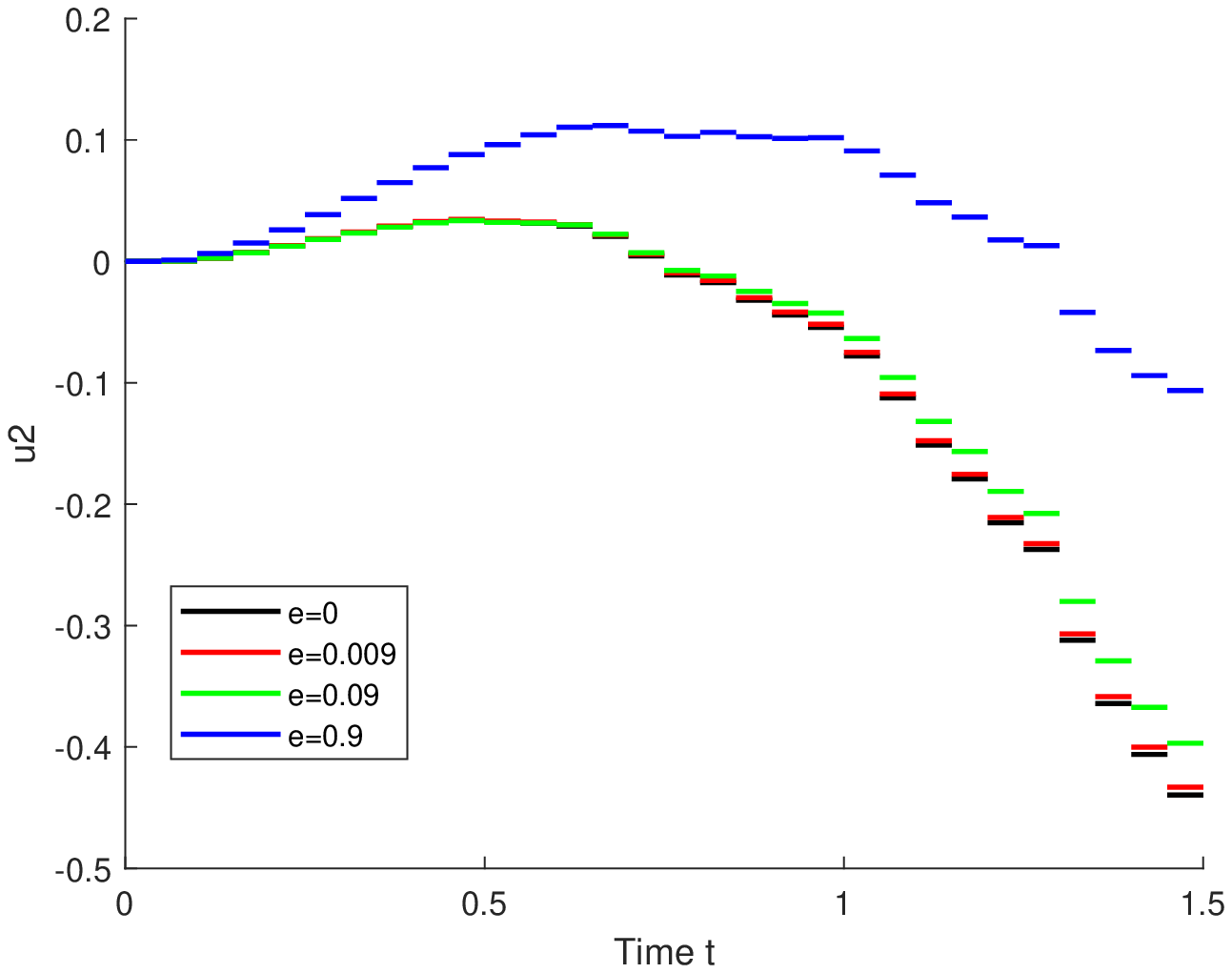}
}\caption{The current through the first diode and the voltage across the second diode change with time in different $\epsilon$. }
\label{ex-1-2-34}
\end{figure}

\begin{figure}[H]
\subfigure[]{
\includegraphics[scale=0.45]{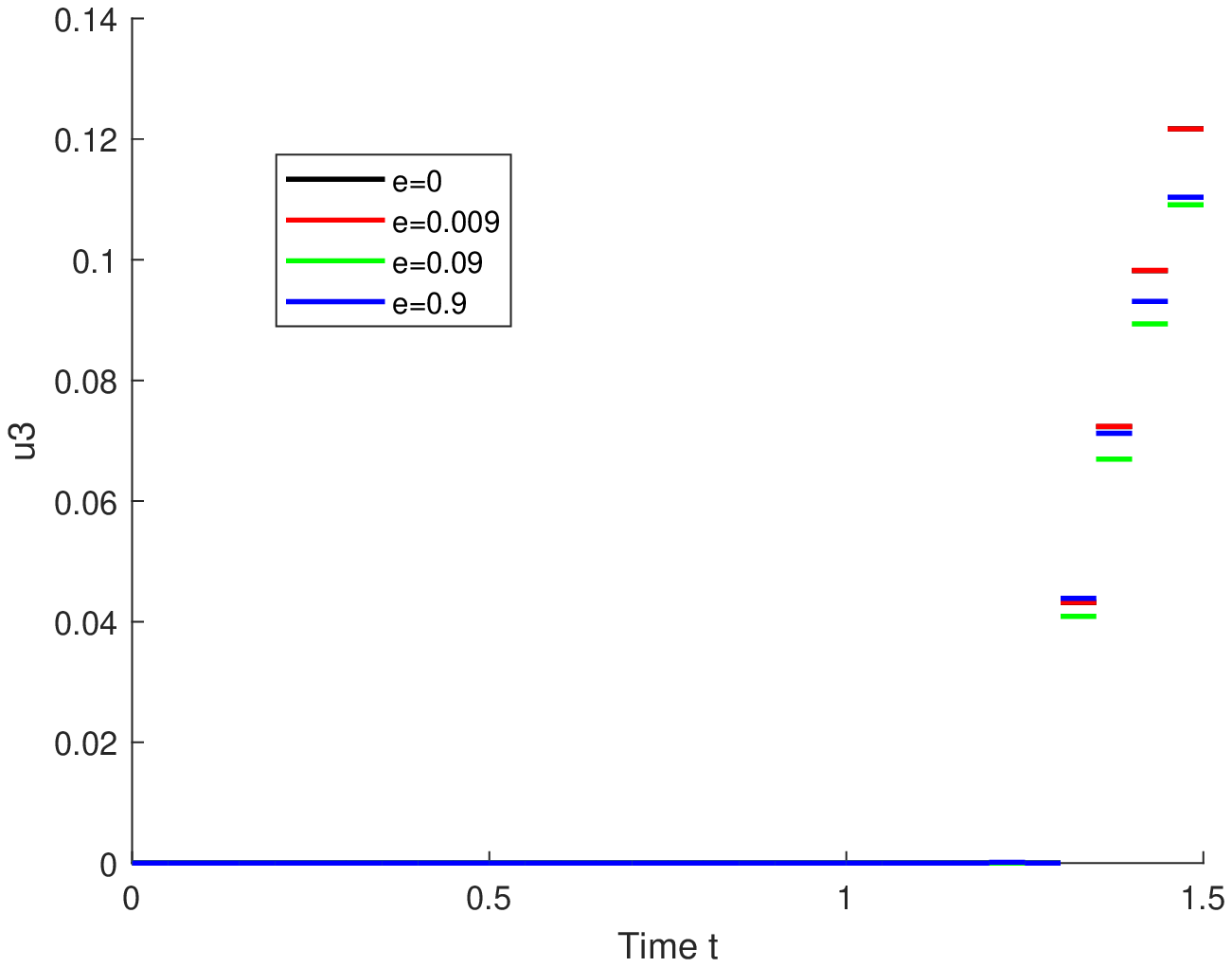}
}
\quad
\subfigure[]{
\includegraphics[scale=0.45]{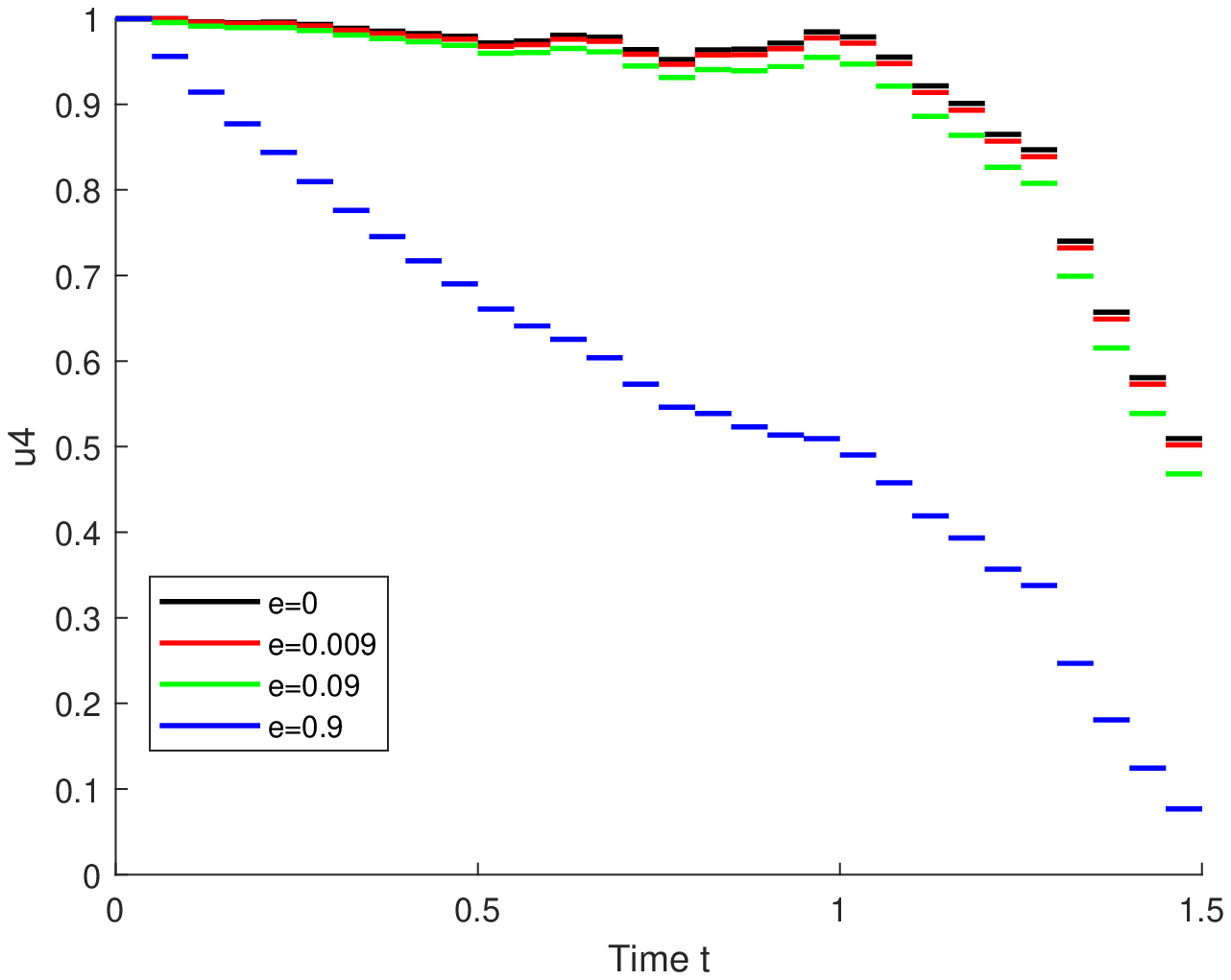}
}
\caption{The voltage across the third diode and the current through the fourth diode change with time in different $\epsilon$.}
\label{ex-1-2-56}
\end{figure}

\subsection{The collapse of the bridge}

In this subsection, we consider an SDVI to model the collapse of the bridge problem in stochastic environment. The following model was introduced in \cite{Chen2012Computational} to describe the collapse of bridge:
\begin{align}\label{ex-2-1}
m\ddot{x}+q(x)=g(x), \quad t\in[0,T],
\end{align}
where $m$ is the mass of the bridge floor, $g$ is the applied force, and
$$
q(x)=
\left\{
\begin{aligned}
&\alpha x, \quad \mbox{if} \quad  x \geq 0, \nonumber \\
&\beta x, \quad \mbox{if} \quad x < 0 \nonumber
\end{aligned}
\right.
$$
is the restoring force for displacement in different directions. Here $\alpha$ and $\beta$ are Hooke's constants for the tension and compression, respectively. From \cite{Ditlevsen1993Plastic,Bensoussan2006Stochastic,Bensoussan2012Asymptotic}, by taking into account the plasticity and white noise in \eqref{ex-2-1}, one has
\begin{align}\label{ex-2-2}
m\ddot{x}+\tau \dot{x} +q(x)=g(x)+k\frac{dB_t}{dt}, \quad t\in[0,T],
\end{align}
where $k\geq 0 $ is the diffusion coefficient to describe the degree of the white noise, $\tau \geq 0$ is the viscous damping coefficient to describe the level of the plasticity of the bridge.
If $\alpha \geq \beta$, then
\begin{align}
q(x)=\alpha x +u,
\end{align}
where $u=\max\left\{0,(\beta-\alpha)x\right\}$.  Thus,
\begin{align}
0\leq u\bot u+(\alpha-\beta)x \geq 0.
\end{align}

Let $y=(y_1,y_2)^T$ with $y_1=x$ and $y_2=\frac{dy_1}{dt}$. Take $\theta \geq 0$, $\alpha=4$, $m=10$, $\beta=1$, and $T=1$. Then \eqref{ex-2-2} can be changed into the following SDVI:
\begin{equation}\label{ex-2}
\left\{
\begin{aligned}
&dy(t)=f(t,y(t),u(t))dt+g(t,y(t),u(t))dB_t,\\
&0\leq u\bot u+3 y_1 \geq 0, \\
&y(0)=(0,\theta)^T,\quad t\in[0,1],
\end{aligned}
\right.
\end{equation}
where
$$
\begin{cases}
f(t,y(t),u(t))=\left(y_2(t),-\frac{2}{5}y_1(t)-\tau y_2(t)-\frac{1}{10}u(t)+sin(4t)\right)^T, \\
g(t,y(t),u(t))=(0,k)^T, \quad \tau \geq 0, \quad k\geq 0, \quad \theta \geq 0.
\end{cases}
$$
Here $\theta$ is initial value of $y_2$ which describes the initial velocity of the oscillation of the bridge.

It follows from Theorem \ref{thm-e} that \eqref{ex-2} admits a unique Carath\'{e}odory solution $(x(t),u(t))$ for fixed $\tau \geq 0$, $k \geq 0$ and $\theta \geq 0$. Moreover, since \eqref{ex-2} satisfies Assumption \ref{a-e}, we can use Algorithm \ref{al} to simulate paths for \eqref{ex-1}. Particularly, take $T=1$, $N=50(h=0.02)$, $x_{0.02}(0)=(0,\theta)^T$ and $u_{0.02}(0)=0$ in Algorithm \ref{al}. By using MATLAB R2018a, we can obtain the numerical results as shown in Figures \ref{ex-2-1-12}-\ref{ex-2-3-3}.

Figures  \ref{ex-2-1-12} and \ref{ex-2-1-3} show the displacement and velocity of the bridge, and the maximum between relative restoring force of the bridge and 0 change with time respectively, for $\tau=1$ and $\theta=0$ in different $k\geq 0$. In fact, Figures  \ref{ex-2-1-12} and \ref{ex-2-1-3} tell us that (i) the curves of the SDVI \eqref{ex-2} and a DVI are coincident when $k=0$, which is drawn by the black line with hollow circles; (ii) the effect of the white noise is very small when the diffusion coefficient $k$ is small and the effect of the white noise can be large when the diffusion coefficient $k$ is big enough, especially the blue line when $k=10$; (iii) the little and appropriate white noise (the green line) can reduce the oscillation of the bridge to some extent and large white noise can impact greatly the oscillation of the bridge. However, it is unknown whether large white noise reduce or increase the oscillation of the bridge in this experiment.

\begin{figure}[H]
\centering
\subfigure[]{
\includegraphics[scale=0.45]{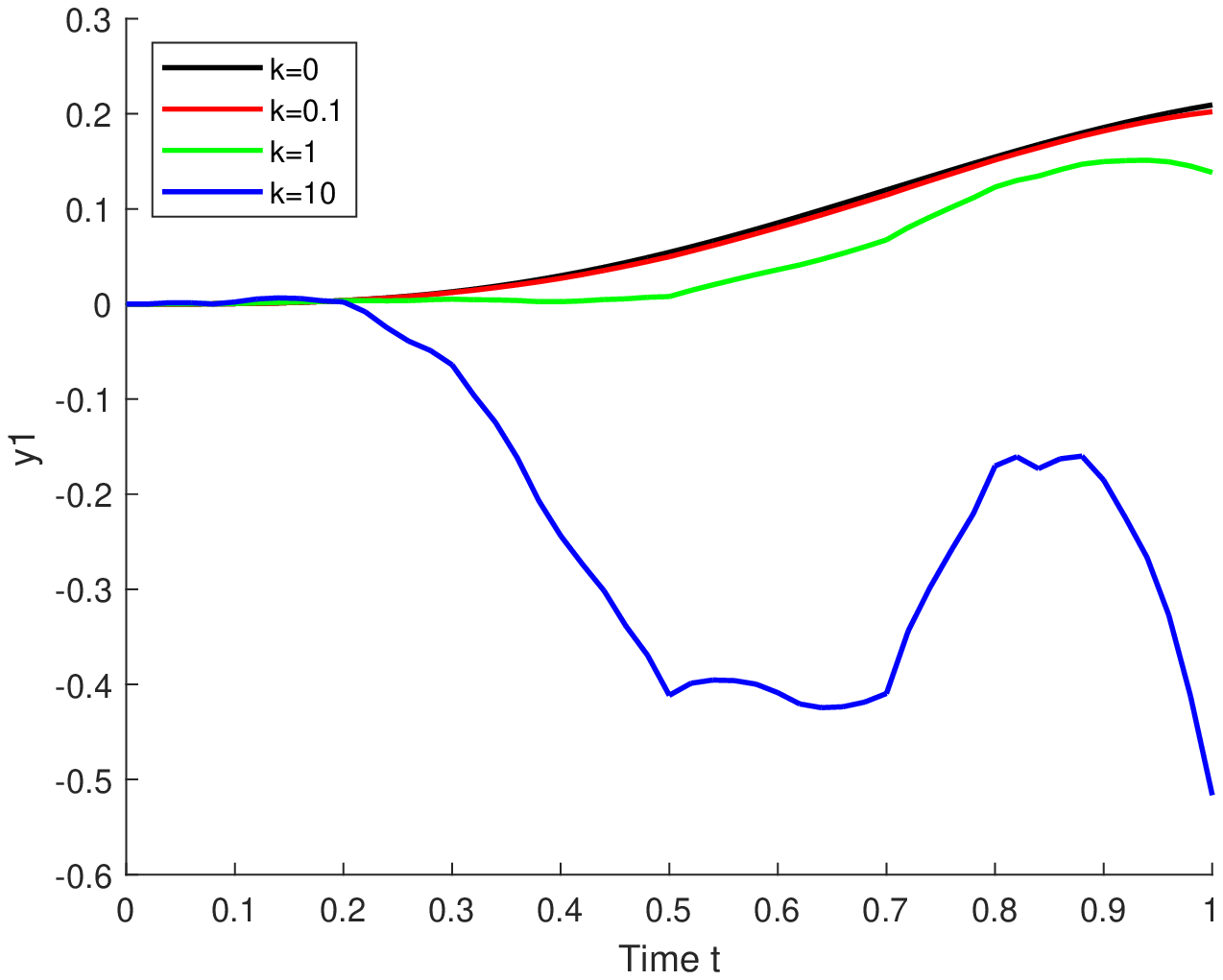}
}
\quad
\subfigure[]{
\includegraphics[scale=0.45]{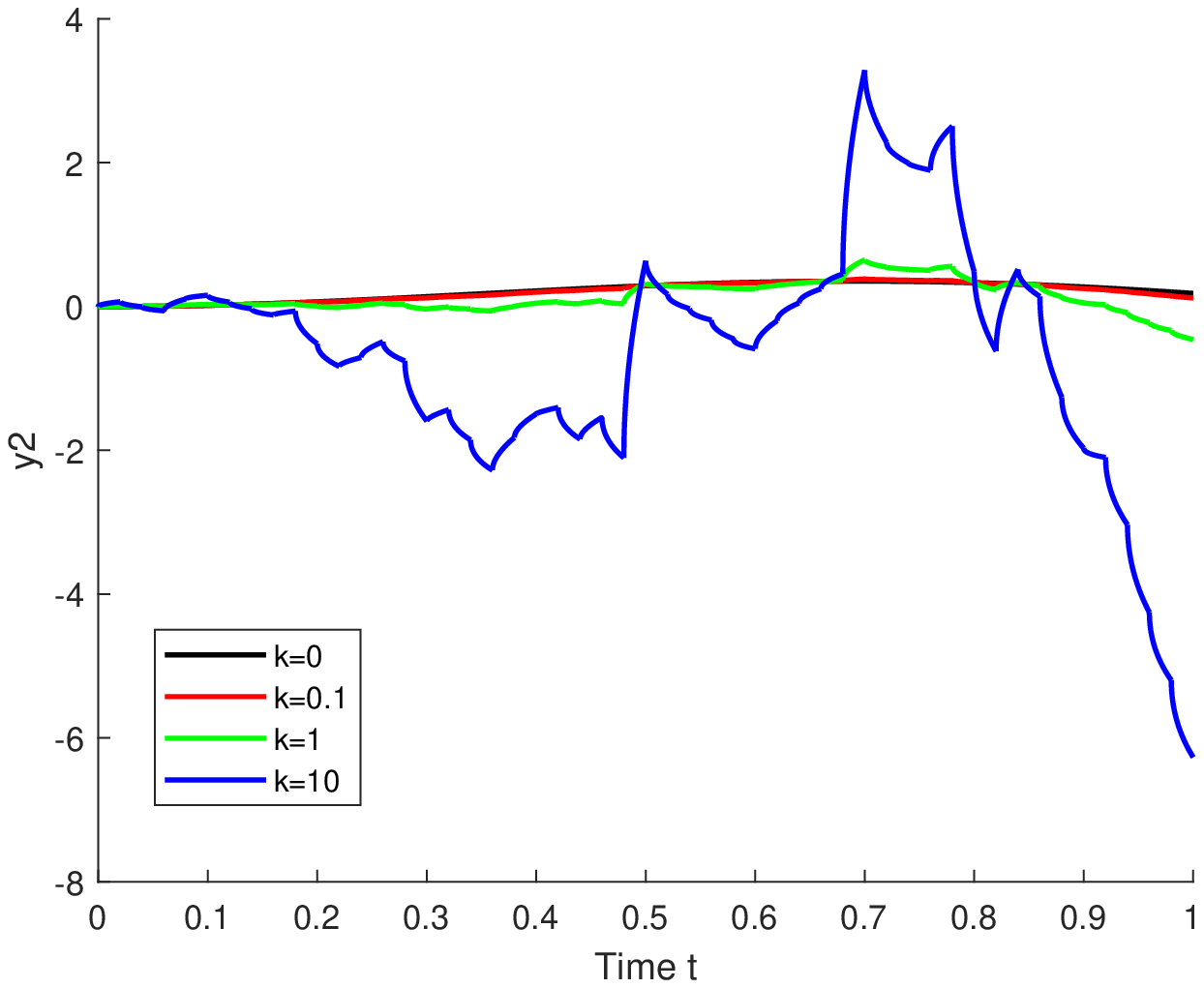}
}
\caption{The displacement and velocity of the bridge change with time in different $k\geq 0$.}
\label{ex-2-1-12}
\end{figure}

\begin{figure}[H]
\centering
\includegraphics[width=0.8\textwidth,height=0.3\textwidth]{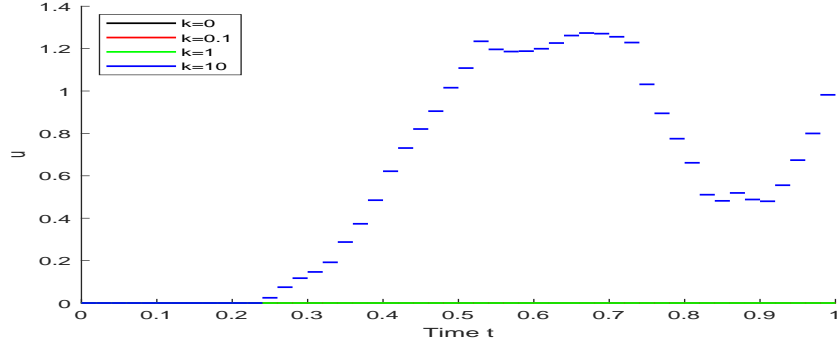}
\caption{The maximum between relative restoring force of the bridge and 0 changes with time in different $k\geq 0$.}
\label{ex-2-1-3}
\end{figure}

On the other hand, Figures  \ref{ex-2-2-12} and \ref{ex-2-2-3} depict the displacement and velocity of the bridge, and the maximum between relative restoring force of the bridge and 0 change with time respectively, for $k=1$ and $\theta=0$ in different $\tau$. Indeed, Figure \ref{ex-2-2-12} tells us that the $\tau$ bigger, the $y_1$ and $y_2$ are smaller, which implies that the high plasticity of the bridge can reduce the oscillation of the bridge, while Figure \ref{ex-2-2-3} displays that the plasticity does not influence $u$.

\begin{figure}[H]
\centering
\subfigure[]{
\includegraphics[scale=0.45]{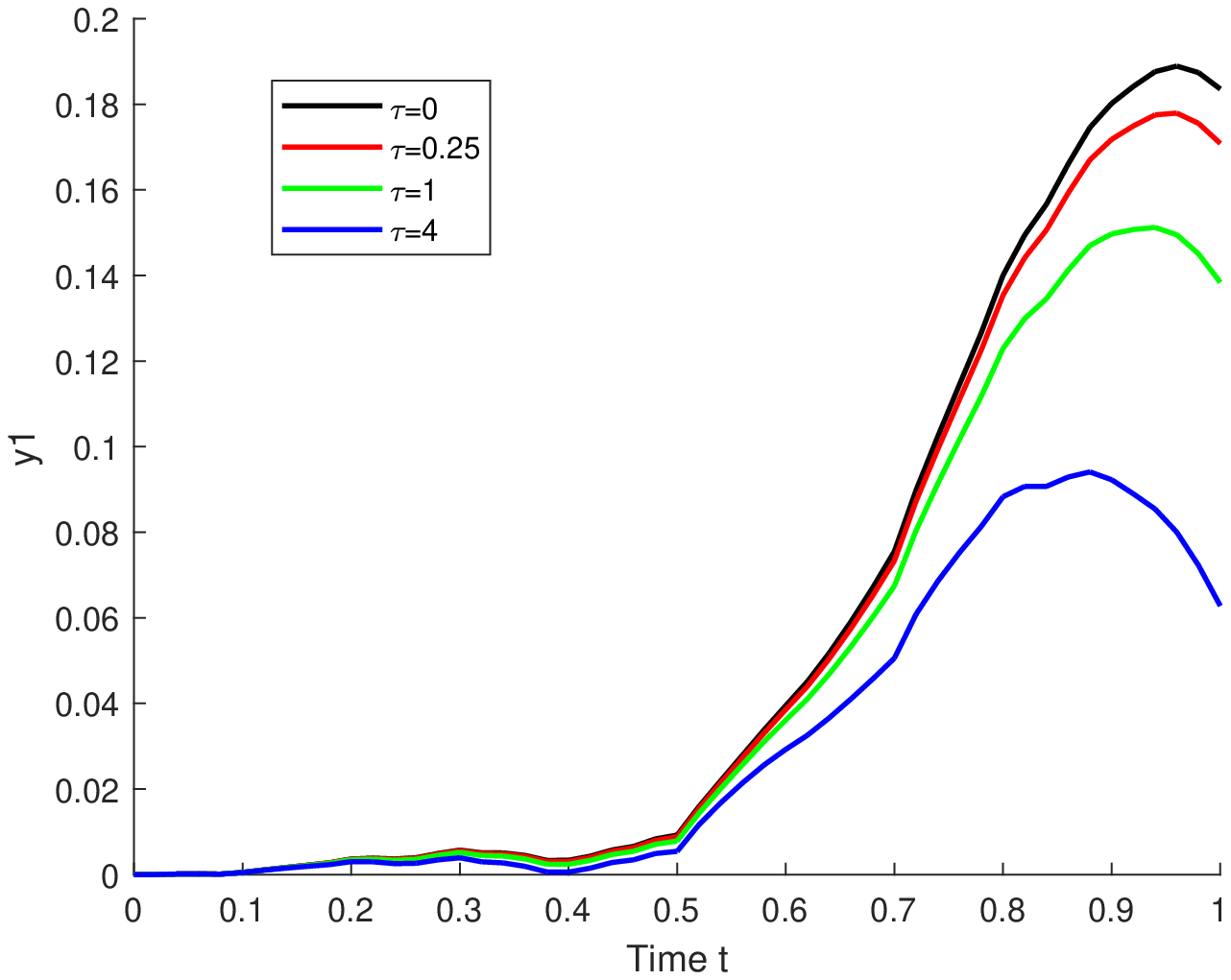}
}
\quad
\subfigure[]{
\includegraphics[scale=0.45]{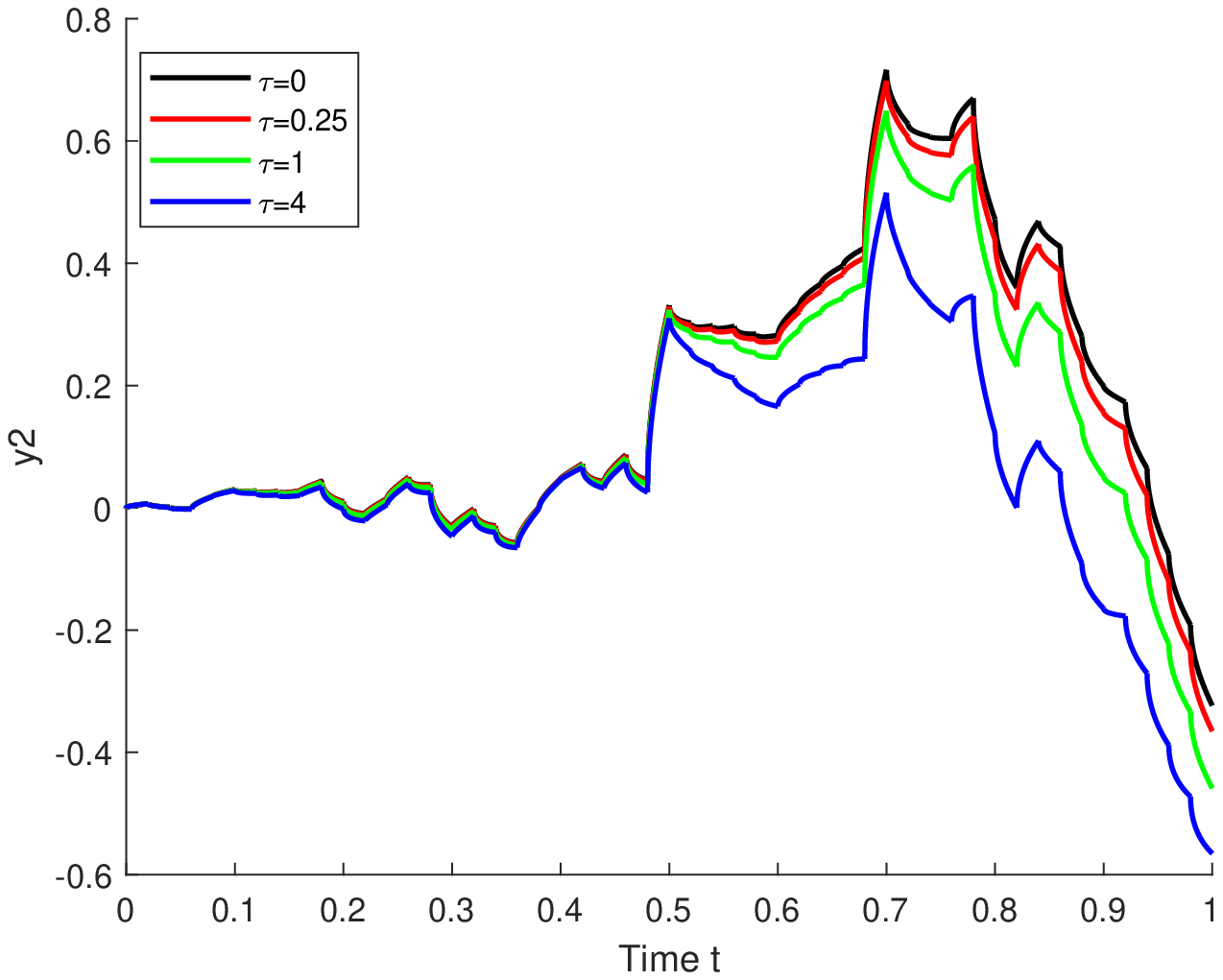}
}
\caption{The displacement and velocity of the bridge change with time in different $\tau$.}
\label{ex-2-2-12}
\end{figure}

\begin{figure}[H]
\centering
\includegraphics[width=0.8\textwidth,height=0.3\textwidth]{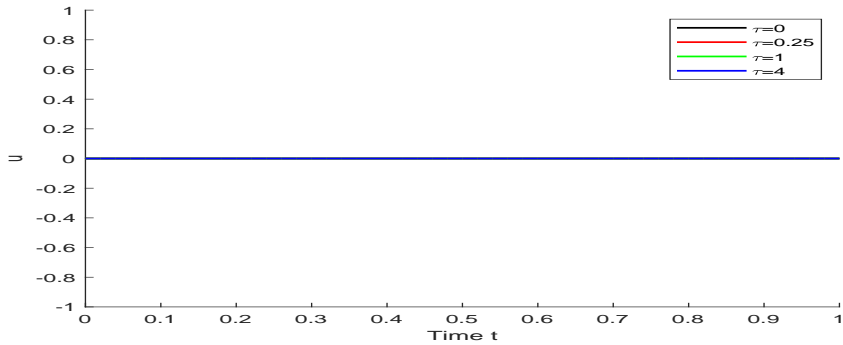}
\caption{The maximum between relative restoring force of the bridge and 0  changes with time in different $\tau$.}
\label{ex-2-2-3}
\end{figure}

Moreover, Figures  \ref{ex-2-3-12} and \ref{ex-2-3-3} show that the displacement and velocity of the bridge, and the maximum between relative restoring force of the bridge and 0 change with time respectively, for $k=1$ and $\tau=1$ in different $\theta$. In fact, Figure \ref{ex-2-3-12} indicates that the $\theta$ smaller, the $y_1$ and $y_2$ are smaller, which implies that the small initial velocity can reduce the oscillation of the bridge, while Figure \ref{ex-2-3-3} shows that the initial velocity does not influence $u$.

\begin{figure}[H]
\centering
\subfigure[]{
\includegraphics[scale=0.45]{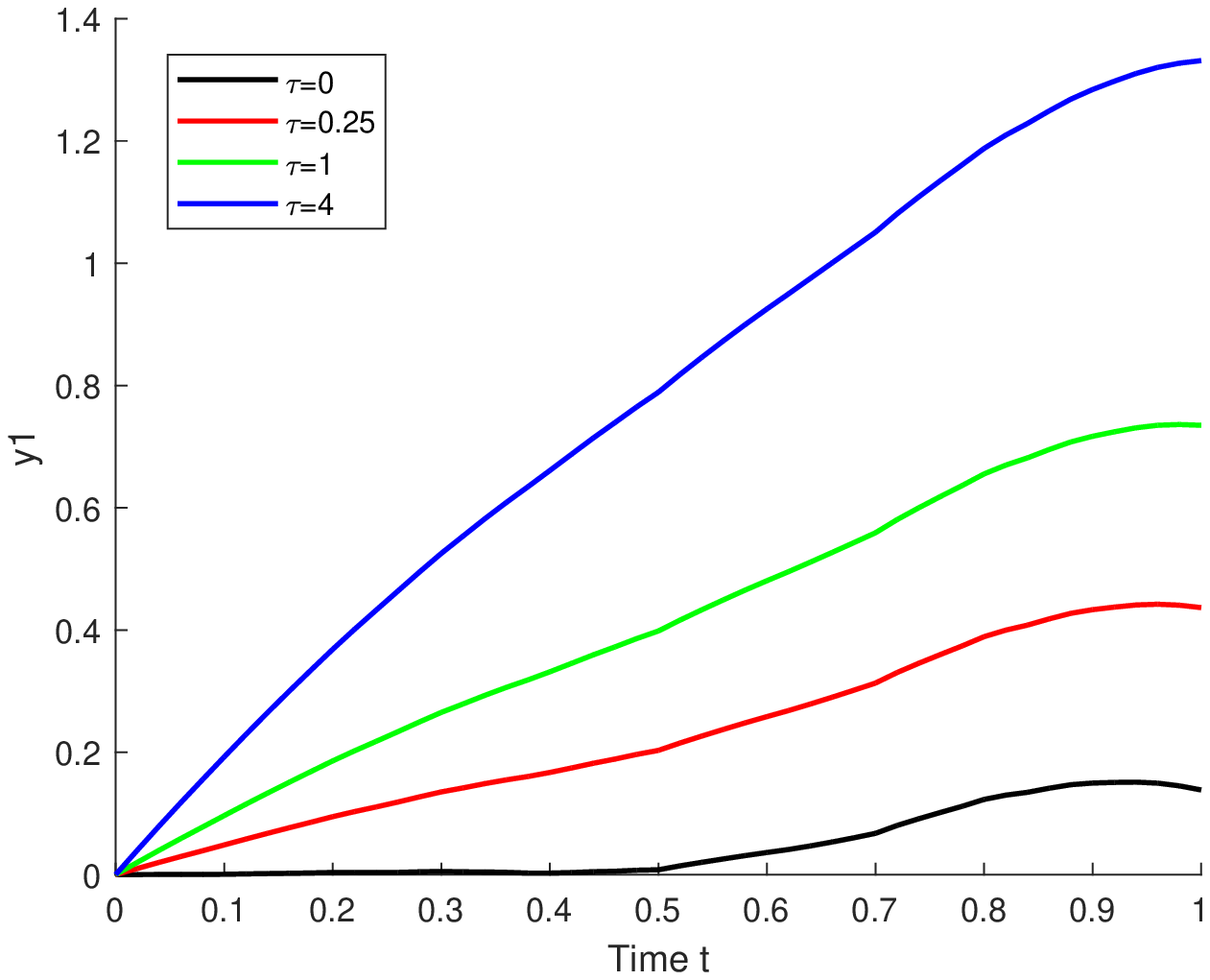}
}
\quad
\subfigure[]{
\includegraphics[scale=0.45]{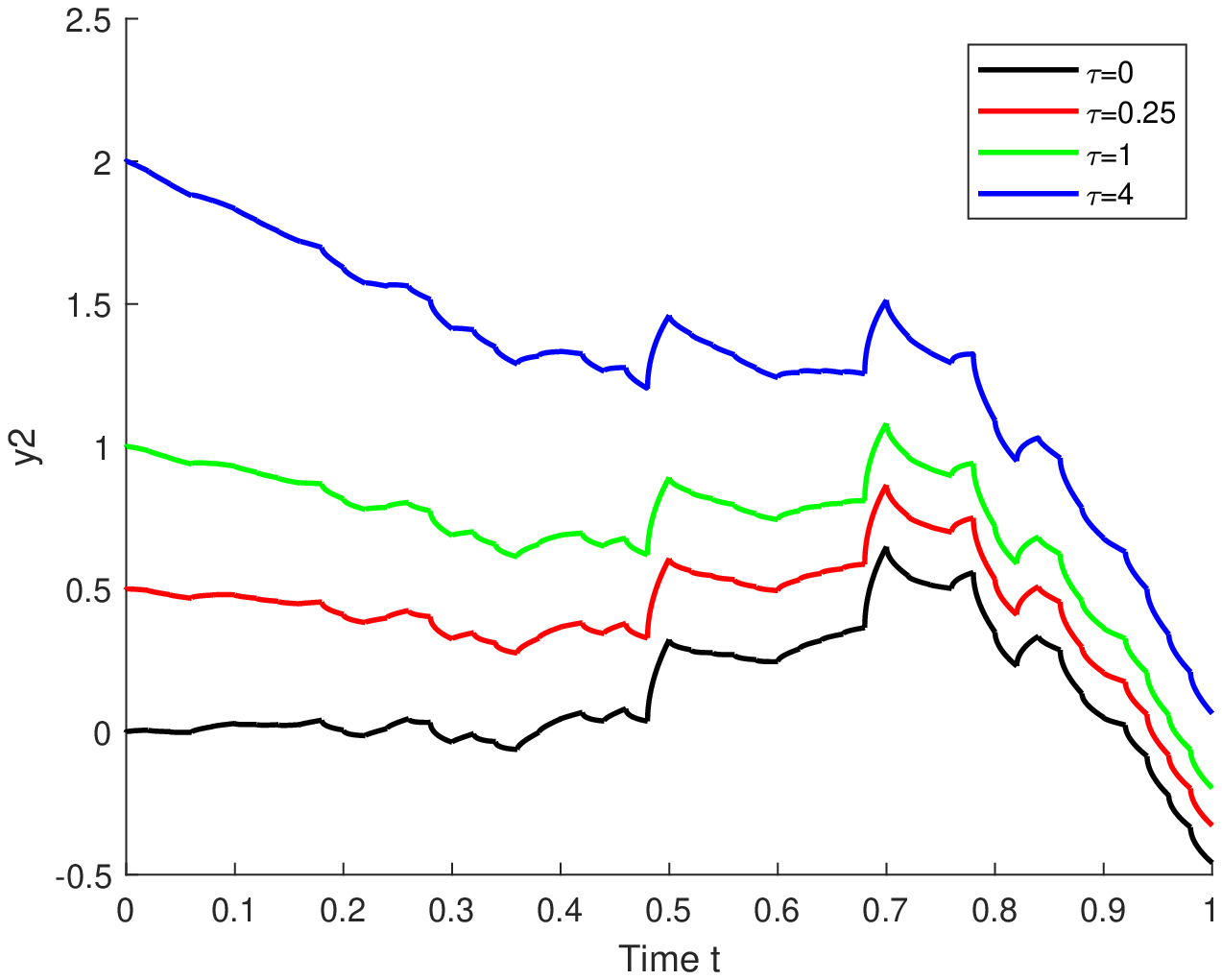}
}
\caption{The displacement and velocity of the bridge change with time in different $\theta$.}
\label{ex-2-3-12}
\end{figure}

\begin{figure}[H]
\centering
\includegraphics[width=0.8\textwidth,height=0.3\textwidth]{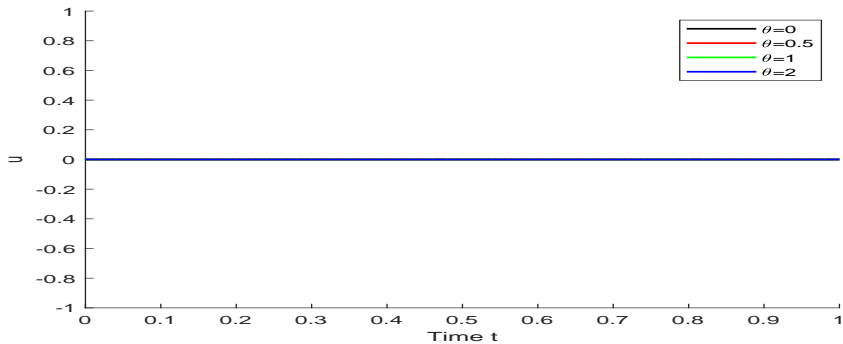}
\caption{The maximum between relative restoring force of the bridge and 0 changes with time in different $\theta$.}
\label{ex-2-3-3}
\end{figure}

The numerical results obtained above illustrate that SDVI is useful and promising for solving some real problems in the stochastic environment.

\end{document}